\documentclass{elsarticle}

\usepackage{amsmath}
\usepackage{amssymb}  
\usepackage{amsthm}
\usepackage{euscript}
\usepackage{todonotes}
\usepackage{comment}

\usepackage{tikz}
\tikzstyle{every picture} = [>=latex]
\usetikzlibrary{backgrounds}
\usetikzlibrary{shapes, shapes.multipart, shapes.geometric}
\usetikzlibrary{arrows,arrows.meta}
\usetikzlibrary{backgrounds}
\usepackage{orcidlink}
\usepackage{cleveref}

\newtheorem{theorem}{Theorem}[section]
\newtheorem{lemma}[theorem]{Lemma}
\newtheorem{proposition}[theorem]{Proposition}
\newtheorem{corollary}[theorem]{Corollary}

\theoremstyle{definition}
\newtheorem{definition}[theorem]{Definition}

\theoremstyle{remark}
\newtheorem{problem}[theorem]{Problem}
\newtheorem{subdefinition}[theorem]{Definition}

\newtheorem{claim}[theorem]{Claim}

\def\ca#1{{\cal#1}}
\def\cw{\mathop{\mbox{\rm cw}}}
\def\hcw#1{\mathop{\text{${#1}$\rm-cw}}}

\begin{document}
\begin{frontmatter}
\myfooter[R]{~\today} 
\title{Hereditary Graph Product Structure and~$\ca H$-clique-width
\footnote{A preliminary short version with weaker results (and incomplete proofs) was published
at the MFCS 2024 conference, \url{https://dx.doi.org/10.4230/LIPIcs.MFCS.2024.61}.}
\footnote{Both authors supported by the project 26-21334S of the Czech Science Foundation.}}

\author{Petr Hlin\v{e}n\'y\corref{corauthor}\,\orcidlink{0000-0003-2125-1514}}
\ead{hlineny@fi.muni.cz}
\author{Jan~Jedelsk\'y\,\orcidlink{0000-0001-9585-2553}}
\ead{xjedelsk@fi.muni.cz}
\address{Faculty of Informatics, Masaryk University,\\Botanick\'a 68a, Brno, Czech Republic}
\cortext[corauthor]{Corresponding author}

\begin{abstract}
We introduce {\em$\ca H$-clique-width}, a new structural measure of graphs that aims to provide a hereditary analogue of the traditional graph product structure.
The definition naturally generalises the ordinary clique-width concept.
As a result, for a class $\ca H$ of graphs (such as the class of paths), the $\ca H$-clique-width of a graph $G$ equals
the least integer $t$ such that $G$ is isomorphic to an induced subgraph of the strong product of a graph from $\ca H$ and a graph of clique-width~$t$.

We study basic properties of $\ca H$-clique-width and compare it to other established structural parameters of graphs.
Notably, we prove that the celebrated Planar graph product structure theorem by Dujmovi\'c et al., and related graph product structure results,
can all be formulated with the {\em induced subgraph} containment relation.
In particular, every planar graph is isomorphic to an induced subgraph of the strong product of a path and a graph of tree-width~$39$.
\end{abstract}
\end{frontmatter}

\section{Introduction}

A prominent structural result by Dujmovi\'c, Joret, Micek, Morin, Ueckerdt and Wood~\cite{DBLP:journals/jacm/DujmovicJMMUW20},
known as the {\em Planar graph product structure theorem},
claims that every planar graph is isomorphic to a subgraph of the strong product ($\boxtimes$) of a path and a graph of small tree-width.
We refer to \Cref{sec:prelim} and \Cref{thm:origprod} for definitions and details, and to \Cref{fig:strongprod}.

The original motivation for this rather recent Planar graph product structure theorem was to bound the queue number of planar graphs,
but the theorem has quickly found interesting applications and follow-up results,
among which we may mention \cite{DBLP:journals/corr/abs-2001-08860,DBLP:journals/jacm/DujmovicEGJMM21,DBLP:journals/siamdm/BonamyGP22,%
Dujmovic2020Planar,DBLP:journals/combinatorics/UeckerdtWY22,DBLP:journals/jctb/DujmovicMW23}.
Namely, the graph product structure theory has been used to study non-repetitive colourings \cite{Dujmovic2020Planar},
to design short labelling schemes~\cite{DBLP:journals/jacm/DujmovicEGJMM21,DBLP:journals/siamdm/BonamyGP22}, 
and to bound the twin-width of planar graphs~\cite{DBLP:journals/corr/abs-2202-11858,DBLP:conf/wg/JacobP22,DBLP:conf/icalp/HlinenyJ23}.

The basic goal of the graph product structure theory can be seen in studying which graph classes admit such a product structure,
that is, for which classes their members are isomorphic to a subgraph of the strong product of two (or more) simpler graphs,
specifically of the product of a path and a graph of small tree-width.
Within this traditional setting, there are two major restrictions; first that the containment (subgraph) relation is not induced,
and second, that graph classes expressible in the described way are necessarily sparse --
they are of bounded local tree-width from the definition.

We propose a different perspective on the graph product structure 
-- the {\em$\ca H$-clique-width} of \Cref{def:hcw} -- addressing both mentioned issues.
On one hand, this perspective allows us to express graphs admitting the traditional product structure as isomorphic to
{\em induced subgraphs} in strong products (\Cref{thm:fromprods0}), bringing a new {\em hereditary graph product structure} concept. 
On the other hand, within this new concept, also classes of {\em dense graphs} can be studied, which is impossible in the traditional product structure theory.
We comment on this improvement in more detail below.

The new concept is closely related to another classical structural notion in graph theory; the clique-width measure (see \Cref{sec:prelim}).
Briefly, the definition of {$\ca H$-clique-width} for a graph class $\ca H$, analogously to the traditional clique-width,
deals with $(H,\ell)$-expressions for $H\in\ca H$ such that every vertex is labelled with a pair $(i,w)$ where $w\in V(H)$ 
and $i\in\{1,\ldots,\ell\}$ is the colour, and the edge-addition operation between colours $i$ and $j$ adds edges precisely between 
vertices labelled $(i,w)$ and $(j,t)$ such that $wt\in E(H)$.  
The $\ca H$-clique-width of a graph $G$ is then the least $\ell$ for which a graph isomorphic to $G$ is built using an $(H,\ell)$-expression.
The full details are presented in \Cref{def:hcw}.

\begin{figure}[t]
$$
\begin{tikzpicture}[scale=0.6]
  \tikzstyle{every node}=[draw, black, shape=circle, minimum size=3pt,inner sep=0pt, fill=none]
  \tikzstyle{every path}=[draw, color=black!30!white]
  \draw (0,1) node[style=ellipse,dashed,fill=blue!10, inner xsep=3.3mm,inner ysep=13mm] {};
  \draw (6.1,5) node[style=ellipse,dashed,fill=blue!10,inner xsep=22mm,inner ysep=3.6mm] {};
  \foreach \x in {0,3,5,7,9}  \foreach \y in {5,3,1,-1} {
    \draw (\x,\y) node[fill=black] {};
  }
  \tikzstyle{every node}=[draw, white, fill=white, shape=circle, minimum size=3pt]
  \draw (0,5) node {};
  \tikzstyle{every path}=[draw, black, thick]
  \draw (0,-1)--(0,3);
  \draw (3,5)--(9,5) (5,5) to[bend left=16] (9,5);
  \draw (3,1)--(3,3) (5,1)--(5,3) (3,3)--(5,3) (3,1)--(5,1) (3,3)--(5,1) (3,1)--(5,3);
  \draw (3,-1)--(3,1) (5,-1)--(5,1) (3,-1)--(5,-1) (3,-1)--(5,1) (3,1)--(5,-1);
  \draw (7,1)--(7,3) (7,3)--(5,3) (7,1)--(5,1) (7,3)--(5,1) (7,1)--(5,3);
  \draw (7,-1)--(7,1) (7,-1)--(5,-1) (7,-1)--(5,1) (7,1)--(5,-1);
  \draw (7,1)--(7,3) (7,3)--(9,3) (7,1)--(9,1) (7,3)--(9,1) (7,1)--(9,3);
  \draw (7,-1)--(7,1) (7,-1)--(9,-1) (7,-1)--(9,1) (7,1)--(9,-1);
  \draw (9,1)--(9,3) (9,3) to[bend right=13] (5,3) (9,1) to[bend right=13] (5,1) (9,3)--(5,1) (9,1)--(5,3);
  \draw (9,-1)--(9,1) (9,-1) to[bend right=13] (5,-1) (9,-1)--(5,1) (9,1)--(5,-1);
\end{tikzpicture} 
\vspace*{-2ex}$$
\caption{Illustrating the strong product $\boxtimes$ of the two shaded graphs.}
\label{fig:strongprod}
\end{figure}

Our proposed alternative view of graph product structure is two-sided.
On one hand, any graph admitting the traditional graph product structure can be obtained as an {\em induced subgraph} of the strong
product of a path and a suitable graph of bounded clique-width, and even of bounded tree-width, too (\Cref{thm:fromprods00}).
On the other hand, a graph $G$ admits the induced product structure with bounded clique-width (of the relevant factor), 
if and only if $G$ has bounded $\ca H$-clique-width where $\ca H$ is the class of reflexive paths (\Cref{thm:hcwproductA}).

The fact that the latter property of having bounded $\ca H$-clique-width can also hold for classes of dense graphs
is particularly useful in the following recent application (by the same authors) in \cite{HtransdarXiv}.
There, it is proved that all transductions (i.e., logical transformations of graphs using the expressive power
of first-order logic) of graph classes admitting the traditional graph product structure must be 
of bounded $\ca H$-clique-width up to bounded perturbations.
This turns out to be a very useful characterisation of such transduced graph classes.
Since transductions of already planar graphs can contain arbitrarily large cliques, in such applications it is necessary
to admit dense graphs and use the induced subgraph relation, which our new concept offers. 

More generally, we believe that being of bounded $\ca H$-clique-width -- where $\ca H$ is the class of reflexive paths as above, 
is a property tightly connected to logic(s) on graphs.
One may view this property as being ``sandwiched'' between the graph classes of bounded clique-width
(which are related to MSO logic on graphs) and the very general monadically dependent graph classes
(which are in relation to FO logic on graphs), and we plan and propose to investigate possible restrictions of MSO logic which
could uncover a similar relation to the graph classes of bounded $\ca H$-clique-width.
Deeply understanding combinatorial properties of $\ca H$-clique-width is a prerequisite for this research direction.

Besides, the generality of our definition and results -- allowing arbitrary graphs of bounded maximum degree (\Cref{thm:fromprods0})
instead of only the paths in the product -- suggests studying $\ca H$-clique-width for various bounded-degree classes~$\ca H$
other than the class of reflexive paths.
For instance, in relation to the aforementioned product-structure studies, one may consider $\ca H$ to be the class
of the graphs $P_n\boxtimes K_c$ (the strong products of paths and the $c$-clique for some fixed~$c$, where a useful
alternative formulation of the Planar graph product structure \cite{DBLP:journals/jacm/DujmovicJMMUW20} uses, e.g.,~$c=3$),
or of $P_n\boxtimes P_m$ as another example.

The presented dual nature of our view of $\ca H$-clique-width and hereditary product structure is another promising enhancement,
and we believe that this view can contribute to finding potential algorithmic applications of the graph product structure theory.

\subsection*{Overview of results}
We start with an overview of our main results which are given here in simplified formulations and using traditional graph-theory terms.
The first one mentioned here ties the $\ca H$-clique-width to strong products of graphs.

\begin{theorem}[\Cref{thm:hcwproduct}]\label{thm:hcwproductA}
Let $\ca H'$ be a family of simple graphs and $\ca H$ be the family of graphs obtained from those of $\ca H'$ by adding a loop to each vertex.
A simple graph $G$ is of $\ca H$-clique-width at most~$\ell$, for an integer $\ell\geq2$, if and only if $G$ is isomorphic to an induced subgraph
of the strong product $H'\boxtimes M$ where $H'\in\ca H'$ and $M$ is a simple graph of clique-width at~most~$\ell$.
\end{theorem}

Further studied basic combinatorial properties of $\ca H$-clique-width include a comparison to the ordinary clique-width and local clique-width.

\begin{theorem}[\Cref{thm:bddiversity}]
There exists a function bounding the $\ca H$-clique-width of any graph $G$ in terms of the ordinary clique-width of~$G$,
if and only if there is an integer $k$ such that every graph in $\ca H$ has at most $k$ vertices with pairwise distinct closed neighbourhoods.
\end{theorem}

\begin{theorem}[simplified~\Cref{thm:localcw}]
Assume that the maximum vertex degree in a class~$\ca H$ is finite, equal to~$\Delta$.
The class of graphs of $\ca H$-clique-width at most $\ell$ is of bounded local clique-width in terms of $\ell$ and~$\Delta$.
\end{theorem}

The core results of our paper extend some of the traditional graph product structure results as follows.

\begin{theorem}[simplified~\Cref{thm:fromprods}]\label{thm:fromprods0}
Let $Q$ be a simple graph of maximum degree $\Delta\geq2$ and $M$ be a simple graph of tree-width~$k$.
Assume that a graph $G$ is a subgraph (not necessarily induced) of the strong product $Q\boxtimes M$, that is, $G\subseteq Q\boxtimes M$.
Then:
\begin{enumerate}[a)]
\item 
There exists a graph $M_1$ of clique-width at most $(\Delta^2+2)\cdot\Delta^{2(\Delta+1)(k+1)}$ such that $G$ is isomorphic to 
an {\em induced} subgraph of the strong product~$Q\boxtimes M_1$; $\>G\subseteq_i Q\boxtimes M_1$.
\item 
There exists a graph $M_2$ of tree-width at most $(k+1)(\Delta^2+1)\cdot\Delta^{2(\Delta+1)(k+1)}$ such that $G$ is isomorphic to 
an {\em induced} subgraph of the strong product~$Q\boxtimes M_2$; $\>G\subseteq_i Q\boxtimes M_2$.
\end{enumerate}
\end{theorem}

Specifically, for the factor $Q$ being a path, we give the following (slightly improved over the previous theorem) bounds:
\begin{theorem}[\Cref{cor:fromprods}]\label{thm:fromprods00}
Assume that a graph $G$ is a subgraph (not necessarily induced) of the strong product $G\subseteq P\boxtimes M$ where $P$ is a path 
and $M$ is a simple graph of tree-width at most~$k$. 
Then there exists a graph $M_1$ of clique-width at most $4\cdot8^{k+1}=2^{3k+5}$,
and a graph $M_2$ of tree-width at most $3(k+1)\cdot8^{k+1}$, such that $G$ is isomorphic to {\em induced} subgraphs
of each of the strong products~$P\boxtimes M_1$ and~$P\boxtimes M_2$.
\end{theorem}

In relation to the Planar graph product structure theorem, \Cref{thm:fromprods00}, using the tree-width bound of $6$ on $M$
given by \cite{DBLP:journals/combinatorics/UeckerdtWY22} (\Cref{thm:origprod}), implies that every planar graph is isomorphic to 
an induced subgraph of the strong product of a path and a graph of tree-width at most $2^{3\cdot6+5}=8\,388\,608$.
We improve the bound in the planar case down to~$39$:

\begin{theorem}[\Cref{thm:onlyplanar}]\label{thm:onlyplanar0}
Every simple planar graph is isomorphic to an \emph{induced} subgraph of the strong product $P\boxtimes M$ where $P$ 
is a path and $M$ is of tree-width at most~$39$.
\end{theorem}

\subsection*{Paper organisation}
Our paper is structured as follows.
In \Cref{sec:prelim}, we introduce the core definitions and new concepts (\Cref{def:hcw} and \Cref{clm:basics}).
In \Cref{sec:propehcw} and \Cref{sec:induprod}, we further study basic properties of $\ca H$-clique-width and characterise 
its relations to ordinary clique-width (\Cref{thm:bddiversity}), to local clique-width (\Cref{thm:localcw}),
to the strong product of graphs (\Cref{thm:hcwproduct}), and in parts to twin-width (\Cref{cor:Hcw-tww}).
These characterisations are provided to complete the picture of $\ca H$-clique-width, and are mostly independent of the
subsequent core product-structure oriented results.

\Cref{sec:indutrad} then proves our core results \Cref{thm:fromprods} and \Cref{cor:fromprods}
(formulated above as \Cref{thm:fromprods0} and \Cref{thm:fromprods00}, respectively).
In \Cref{sec:induplanar}, we specifically study the induced product structure for planar graphs -- \Cref{thm:onlyplanar}
(formulated above as \Cref{thm:onlyplanar0}).
We conclude with some open combinatorial questions related to the new concept of $\ca H$-clique-width in \Cref{sec:conclu}.

\section{Preliminaries}\label{sec:prelim}

We consider finite simple graphs, i.e., graphs without parallel edges or loops, but in one specific context (\Cref{def:hcw})
we allow graphs with simple edges and optional \mbox{(self-)}loops, thereafter called {loop graphs}.
Precisely, a \emph{loop graph} is a multigraph not allowing parallel edges, but allowing at most one loop per vertex.
In the context of loop graphs, we specially call a graph $G$ a {\em reflexive (loop) graph} if every vertex of $G$ has a loop.
We naturally use the terms \emph{reflexive path}, \emph{reflexive clique}, and \emph{reflexive independent set}
to denote ordinary paths, cliques, and independent sets, respectively, with loops added to all vertices.
We denote by $G[X]$ the subgraph of $G$ induced on the vertex set~$X\subseteq V(G)$,
and write $G_1\subseteq_iG_2$ to say that $G_1$ is an induced subgraph of~$G_2$.
Note that if $G_1\subseteq_iG_2$ for loop graphs, then possible loops of $G_2$ on vertices of $G_1$ are also inherited.

A graph $G$ is a {\em matching} if $G$ is simple and all vertex degrees in $G$ are $1$.
A graph $G\subseteq K_{n,n}$ is a {\em co-matching} if $G$ is obtained from $K_{n,n}$ by removing the edge set of a matching of~$n$ edges.
A graph $G$ is a {\em half-graph} if $G$ is a bipartite simple graph with the bipartition
$\{u_1,\ldots,u_n\}$ and $\{v_1,\ldots,v_n\}$, such that $u_iv_j\in E(G)$ if and only if~$i\leq j$.
A bipartite graph $G$ with a fixed bipartition $V(G)=A\cup B$ is a {\em bi-induced subgraph} of a graph $H$, 
if $G\subseteq H$ such that every edge of $H$ with one end in $A$ and the other end in $B$ is present in~$G$.
A~bipartite graph~$G$ (without fixing its bipartition) is a {\em bi-induced subgraph} of $H$, 
if there exists a bipartition $V(G)=A\cup B$ into independent sets $A,B$ such that $G$ with $(A,B)$ is a bi-induced subgraph of~$H$.
Note that a bi-induced subgraph of~$H$ is not necessarily an induced subgraph of~$H$.

\subparagraph{Graph product structure theory. }
The {\em strong product} $G_1\boxtimes G_2$ of two simple graphs is the graph $G$ on the vertex set $V(G):=V(G_1)\times V(G_2)$
such that, for any $[u,u'],\,[v,v']\in V(G)$,%
\footnote{Please note that we specifically use square brackets, such as in $[u,u']$, to refer to the vertices
of the strong product in the realm of product structure theory.
The reason is to visually emphasise the product vertices in context where other ordered tuples are used, e.g., in place of the vertex $u$ or $u'$ in $[u,u']$.}
we have $\{[u,u'],\,[v,v']\}\in E(G)$ if and only if
($uv\in E(G_1)$ and $u'v'\in E(G_2)$) or ($u=v$  and $u'v'\in E(G_2)$) or ($uv\in E(G_1)$ and $u'=v'$).

For an illustration, see \Cref{fig:strongprod},
or notice that the strong product $P\boxtimes Q$ of any two paths $P,Q$ is the square grid with diagonals.
It may be interesting to observe that, in the context of loop graphs, if both $G_1$ and $G_2$ are reflexive,
then the definition of the strong product $G_1\boxtimes G_2$ could be shortened as ``$uv\in E(G_1)$ and $u'v'\in E(G_2)$,''
and the result would be the same except that all vertices would have loops.

The origins of {graph product structure theory} go back to the aforementioned seminal paper~\cite{DBLP:journals/jacm/DujmovicJMMUW20}:
\begin{theorem}[$\!$\cite{DBLP:journals/jacm/DujmovicJMMUW20}, improved in \cite{DBLP:journals/combinatorics/UeckerdtWY22}]\label{thm:origprod}
Every planar graph is isomorphic to a subgraph of the strong product $P\boxtimes M$ where $P$ is a path and $M$ is a planar graph of tree-width at most~$6$.
\end{theorem}

We remark that alternative refined formulations of \Cref{thm:origprod} have been published;
most remarkably, expressing planar graphs as isomorphic to subgraphs of the strong product
$P\boxtimes K_3\boxtimes M$ where $P$ is again a path and $M$ is of tree-width at most~$3$ \cite{DBLP:journals/jacm/DujmovicJMMUW20}.
The latter formulation is of importance in several applications, such as in refining the upper bound on the queue number of planar graphs.
Furthermore, \Cref{thm:origprod} has been extended to other graph classes, such as 
to graphs embeddable on a fixed surface \cite{DBLP:journals/jacm/DujmovicJMMUW20,DBLP:conf/mfcs/KralPS24},
to powers of bounded-degree embeddable graphs and $k$-planar graphs \cite{DBLP:journals/jctb/DujmovicMW23,distel2023powers},
and to so-called $h$-framed graphs \cite{DBLP:conf/isaac/BekosLH022}.

Our goal is to enhance product-structure results with the {\em induced subgraph} containment, as formulated above in \Cref{thm:fromprods0}.
To achieve this goal, we introduce the new concept of $\ca H$-clique-width in \Cref{def:hcw} and study its properties.

\subparagraph{Width measures.}
As a traditional structural decomposition, a {\em tree decomposition} of a graph $G$ is a tuple $(T,\ca X)$ where $T$ is a tree
and $\ca X=\{X_t:t\in V(T)\}$ where $X_t\subseteq V(G)$ are called {\em bags}. The bags have to satisfy the following:
(i) $\bigcup_{t\in V(T)}X_t=V(G)$,
(ii) for every vertex $v\in V(G)$, the set of the nodes $t\in V(T)$ such that $v\in X_t$ forms a subtree in~$T$,
and (iii) for every edge $uv\in E(G)$, there exists $t\in V(T)$ such that $\{u,v\}\subseteq X_t$.
The {\em tree-width} of $G$ is the minimum of $\max_{t\in V(T)}|X_t|-1$ over all tree decompositions of~$G$.

Our work is closely related to another measure, which is a ``dense counterpart'' of tree-width.
The {\em clique-width} of a graph $G$ is the minimum integer $\ell$ such that~$G$ (irrespective of labelling)
is the value of an algebraic {\em$\ell$-expression} defined by the following operations:
\begin{itemize}
\item create a new vertex of label (colour) $i$ for some $i\in\{1,\ldots,\ell\}$;
\item take the disjoint union of two labelled graphs;
\item for $1\leq i\not=j\leq\ell$, add all (missing) edges between a vertex of label $i$ and a vertex of label~$j$;
\item for $1\leq i\not=j\leq\ell$, recolour each vertex of label $i$ to have label~$j$.
\end{itemize}

In the same direction, let the {\em local clique-width} of a graph $G$ be the integer function $\lambda$ defined as follows;
for an integer distance $r\geq1$,  $\lambda(r)$ is the maximum clique-width of the $r$-neighbourhood of a vertex $v$ in $G$, over all~$v\in V(G)$.
We say that a graph class $\ca G$ is of {\em bounded local clique-width} if there exists an integer function upper-bounding
the local clique-width of every member of~$\ca G$.
For instance, the class of grids is of bounded local clique-width, but of unbounded clique-width.

Another well-known measure is the bandwidth.
For a graph $G$, consider a linear ordering of its vertices $V(G)=(v_1,v_2,\ldots,v_n)$.
The bandwidth of this ordering is the maximum of $|i-j|$ over all $\{v_i,v_j\}\in E(G)$,
and the {\em bandwidth} of~$G$ is the minimum bandwidth over all linear orderings of~$V(G)$.
Notice that a bound on the bandwidth of a simple graph implies a bound on its maximum degree, but the converse is far from being true.

The last measure we mention, twin-width, was introduced a few years ago by Bonnet et al.\ in \cite{DBLP:journals/jacm/BonnetKTW22}.
A \emph{trigraph} is a simple graph $G$ in which some edges are marked as {\em red}, and with respect to the red edges only, 
we naturally speak about \emph{red neighbours} and \emph{red degree} in~$G$. 
For a pair of (possibly not adjacent) vertices $x_1,x_2\in V(G)$, we define a {contraction} of the pair $x_1,x_2$ as the operation
creating a trigraph $G'$ which is the same as $G$ except that $x_1,x_2$ are replaced with a new vertex $x_0$ (said to {\em stem from $x_1,x_2$}) such that:
\begin{itemize}
\item 
the (full) neighbourhood of $x_0$ in $G'$ (i.e., including the red neighbours), denoted by $N_{G'}(x_0)$,
equals the union of the neighbourhoods $N_G(x_1)$ of $x_1$ and $N_G(x_2)$ of $x_2$ in $G$ except $x_1,x_2$ themselves, that is,
$N_{G'}(x_0)=(N_G(x_1)\cup N_G(x_2))\setminus\{x_1,x_2\}$, and
\item 
the red neighbours of $x_0$, denoted here by $N_{G'}^r(x_0)$, inherit all red neighbours of $x_1$ and of $x_2$ and add those in $N_G(x_1)\Delta N_G(x_2)$,
that is, $N_{G'}^r(x_0)=\big(N_{G}^r(x_1)\cup N_G^r(x_2)\cup(N_G(x_1)\Delta N_G(x_2))\big)\setminus\{x_1,x_2\}$, where $\Delta$ denotes the~symmetric set difference.
\end{itemize}
A \emph{contraction sequence} of a trigraph $G$ is a sequence of successive contractions turning~$G$  into a single vertex,
and its {width} $d$ is the maximum red degree of any vertex in any trigraph of the sequence.
The \emph{twin-width} of a trigraph $G$ is the minimum width over all possible contraction sequences of~$G$.

\subparagraph{Introducing $\ca H$-clique-width. }
Our main contribution builds on the following new concept.

\begin{definition}[$\ca H$-clique-width]\label{def:hcw}
Let $\ca H$ be a family of loop graphs, and $\ell>0$ be an integer.
Consider {\em labels} of the form $(i,v)$ where $i\in\{1,\ldots,\ell\}$ and $v\in V(H)$ for some (fixed)~$H\in\ca H$.
\begin{enumerate}[a)]
\item For $H\in\ca H$, let an {\em$(H,\ell)$-expression} be an algebraic expression
using the following four operations on vertex-labelled graphs:
\begin{itemize}
\item creating a new vertex with a single label $(i,v)$ for some $i\in\{1,\ldots,\ell\}$ and $v\in V(H)$;
\item taking the disjoint union of two labelled graphs;
\item\label{it:Hcwedges}
for $1\leq i\not=j\leq\ell$, \emph{adding edges between $i$ and $j$}, which means to
add all edges between the vertices of label $(i,v)$ and the vertices of label $(j,w)$
over all pairs $(v,w)\in V(H)\times V(H)$ such that $vw\in E(H)$;\,%
\footnote{Notice that this condition includes the case of the same vertex $v=w$ with a loop in~$H$,
which will often be assumed to exist in our graphs~$H$.}
and
\item for $1\leq i\not=j\leq\ell$, \emph{recolouring $i$ to $j$}, which means to
relabel all vertices with label $(i,v)$ where $v\in\!V(H)$ to~label~$(j,v)$.
\end{itemize}
\smallskip
\item The \emph{$\ca H$-clique-width} $\hcw{\ca H}(G)$ of a simple graph $G$ is defined as the smallest integer $\ell$
such that (some labelling of) $G$ is the value of an $(H,\ell)$-expression for some $H\in\ca H$.
If it is not possible to build $G$ this way, then let $\hcw{\ca H}(G)=\infty$.
\end{enumerate}
Given an $(H,\ell)$-expression of value (a labelled graph)~$G$, 
we use the following terminology; the graph $H$ is the {\em parameter} of the expression,
and when referring to a label $(i,v)$ of $x\in V(G)$, the integer $i$ is the {\em colour} and $v$ the {\em parameter vertex} of~$x$.
\end{definition}
Observe that, throughout an $(H,\ell)$-expression $\varphi$ valued $G$, the colours of the vertices of $G$ may arbitrarily change by
the recolouring operations, but the parameter vertex of every $x\in V(G)$ stays the same (is uniquely determined for~$x$) in~$\varphi$.

\Cref{def:hcw} is illustrated in \Cref{fig:hcw-grid}.

It is obvious that $\ca H$-clique-width, similarly to ordinary clique-width, is monotone under taking induced subgraphs.
On the other hand, unlike ordinary clique-width, $\ca H$-clique-width does not seem to be functionally closed under taking
the complement of a graph in any reasonably general sense.
Another remark concerns the family $\ca H$ which should be generally treated as an infinite class of (finite) loop graphs,
due to $\ca H$-clique-width being asymptotically the same as ordinary clique-width in the case of finite~$\ca H$
-- see \Cref{thm:bddiversity}.

\begin{figure}[tp]
\def\xscal{0.8}
\def\yscal{0.65}
\begin{tikzpicture}[xscale=\xscal, yscale=\yscal]
  \tikzstyle{every node}=[draw, black, shape=circle, minimum size=3pt,inner sep=0.5pt, fill=gray]
  \tikzstyle{every path}=[draw, color=gray]
  \foreach \y in {4,3,2,1} {
    \draw (0,\y) node[label=left:$v_\y$] {} to[out=90,in=90] ++(0.33,0) to[out=270,in=270] ++(-0.33,0);
  } \draw (0,1)--(0,4);
  \node[draw=none,fill=none] at (0,0) {$P\in\ca H$};
  \tikzstyle{every path}=[draw, color=black, thick]\small
  \draw (3,1) node[fill=red!50!white] {$v_1$};
  \draw (3,2) node[fill=blue!30!white] {$v_2$};
  \node[draw=none,fill=none] at (2.3,0.5) {(a)};
  \draw (5,1)--(5,2);
  \draw (5,1) node[fill=red!50!white] {$v_1$};
  \draw (5,2) node[fill=blue!30!white] {$v_2$};
  \node[draw=none,fill=none] at (4.3,0.5) {(b)};
  \draw (7,1)--(7,2);
  \draw (7,1) node[fill=red!50!white] {$v_1$};
  \draw (7,2) node[fill=red!20!white] {$v_2$};
  \draw (7,3) node[fill=blue!30!white] {$v_3$};
  \node[draw=none,fill=none] at (6.3,0.5) {(c)};
  \draw (9,1)--(9,3);
  \draw (9,1) node[fill=red!50!white] {$v_1$};
  \draw (9,2) node[fill=red!20!white] {$v_2$};
  \draw (9,3) node[fill=red!50!white] {$v_3$};
  \draw (9,4) node[fill=blue!30!white] {$v_4$};
  \node[draw=none,fill=none] at (8.3,0.5) {(d)};
  \draw (11,1)--(11,4);
  \foreach \y in {4,2} \draw (11,\y) node[fill=red!20!white] {$v_\y$};
  \foreach \y in {3,1} \draw (11,\y) node[fill=red!50!white] {$v_\y$};
  \node[draw=none,fill=none] at (11,0) {$P_1$};
  \node[draw=none,fill=none] at (10.3,0.5) {(e)};
\end{tikzpicture} 

\begin{tikzpicture}[xscale=\xscal, yscale=\yscal]
  \tikzstyle{every node}=[draw, black, shape=circle, minimum size=3pt,inner sep=0.5pt, fill=gray]
  \tikzstyle{every path}=[draw, color=gray]
  \foreach \y in {4,3,2,1} {
    \draw (0,\y) node[label=left:$v_\y$] {} to[out=90,in=90] ++(0.33,0) to[out=270,in=270] ++(-0.33,0);
  } \draw (0,1)--(0,4);
  \node[draw=none,fill=none] at (0,0) {$P\in\ca H$};
  \tikzstyle{every path}=[draw, color=black, thick]\small
  \draw (3,1)--(3,4);
  \foreach \y in {4,2} \draw (3,\y) node[fill=red!20!white] {$v_\y$};
  \foreach \y in {3,1} \draw (3,\y) node[fill=red!50!white] {$v_\y$};
  \node[draw=none,fill=none] at (3,0) {$P_1$};
  \draw (4,1)--(4,4);
  \foreach \y in {4,2} \draw (4,\y) node[fill=green!25!white] {$v_\y$};
  \foreach \y in {3,1} \draw (4,\y) node[fill=green!60!lightgray] {$v_\y$};
  \node[draw=none,fill=none] at (4,0) {$P_2$};
  \node[draw=none,fill=none] at (2.1,0.5) {(f)};
  \foreach \y in {3,1} \draw (7,\y)--++(1,0);
  \draw (7,1)--(7,4);
  \foreach \y in {4,2} \draw (7,\y) node[fill=red!20!white] {$v_\y$};
  \foreach \y in {3,1} \draw (7,\y) node[fill=red!50!white] {$v_\y$};
  \node[draw=none,fill=none] at (7,0) {$P_1$};
  \draw (8,1)--(8,4);
  \foreach \y in {4,2} \draw (8,\y) node[fill=green!25!white] {$v_\y$};
  \foreach \y in {3,1} \draw (8,\y) node[fill=green!60!lightgray] {$v_\y$};
  \node[draw=none,fill=none] at (8,0) {$P_2$};
  \node[draw=none,fill=none] at (6.1,0.5) {(g)};
  \foreach \y in {4,3,2,1} \draw (11,\y)--++(1,0);
  \draw (11,1)--(11,4);
  \foreach \y in {4,2} \draw (11,\y) node[fill=red!20!white] {$v_\y$};
  \foreach \y in {3,1} \draw (11,\y) node[fill=red!50!white] {$v_\y$};
  \node[draw=none,fill=none] at (11,0) {$P_1$};
  \draw (12,1)--(12,4);
  \foreach \y in {4,2} \draw (12,\y) node[fill=green!25!white] {$v_\y$};
  \foreach \y in {3,1} \draw (12,\y) node[fill=green!60!lightgray] {$v_\y$};
  \node[draw=none,fill=none] at (12,0) {$P_2$};
  \node[draw=none,fill=none] at (10.1,0.5) {(h)};
\end{tikzpicture} 

\begin{tikzpicture}[xscale=\xscal, yscale=\yscal]
  \tikzstyle{every node}=[draw, black, shape=circle, minimum size=3pt,inner sep=0.5pt, fill=gray]
  \tikzstyle{every path}=[draw, color=gray]
  \foreach \y in {4,3,2,1} {
    \draw (0,\y) node[label=left:$v_\y$] {} to[out=90,in=90] ++(0.33,0) to[out=270,in=270] ++(-0.33,0);
  } \draw (0,1)--(0,4);
  \node[draw=none,fill=none] at (0,0) {$P\in\ca H$};
  \tikzstyle{every path}=[draw, color=black, thick]\small
  \foreach \y in {4,3,2,1} \draw (3,\y)--++(1,0);
  \draw (3,1)--(3,4);
  \foreach \y in {4,2} \draw (3,\y) node[fill=blue!30!white] {$v_\y$};
  \foreach \y in {3,1} \draw (3,\y) node[fill=blue!30!white] {$v_\y$};
  \node[draw=none,fill=none] at (3,0) {$P_1$};
  \draw (4,1)--(4,4);
  \foreach \y in {4,2} \draw (4,\y) node[fill=green!25!white] {$v_\y$};
  \foreach \y in {3,1} \draw (4,\y) node[fill=green!60!lightgray] {$v_\y$};
  \node[draw=none,fill=none] at (4,0) {$P_2$};
  \draw (5,1)--(5,4);
  \foreach \y in {4,2} \draw (5,\y) node[fill=red!20!white] {$v_\y$};
  \foreach \y in {3,1} \draw (5,\y) node[fill=red!50!white] {$v_\y$};
  \node[draw=none,fill=none] at (5,0) {$P_3$};
  \node[draw=none,fill=none] at (2.1,0.5) {(i)};
  \foreach \y in {4,3,2,1} \draw (7,\y)--++(2,0);
  \draw (7,1)--(7,4);
  \foreach \y in {4,2} \draw (7,\y) node[fill=blue!30!white] {$v_\y$};
  \foreach \y in {3,1} \draw (7,\y) node[fill=blue!30!white] {$v_\y$};
  \node[draw=none,fill=none] at (7,0) {$P_1$};
  \draw (8,1)--(8,4);
  \foreach \y in {4,2} \draw (8,\y) node[fill=green!25!white] {$v_\y$};
  \foreach \y in {3,1} \draw (8,\y) node[fill=green!60!lightgray] {$v_\y$};
  \node[draw=none,fill=none] at (8,0) {$P_2$};
  \draw (9,1)--(9,4);
  \foreach \y in {4,2} \draw (9,\y) node[fill=red!20!white] {$v_\y$};
  \foreach \y in {3,1} \draw (9,\y) node[fill=red!50!white] {$v_\y$};
  \node[draw=none,fill=none] at (9,0) {$P_3$};
  \node[draw=none,fill=none] at (6.1,0.5) {(j)};
  \foreach \y in {4,3,2,1} \draw (11,\y)--++(2,0);
  \draw (11,1)--(11,4);
  \foreach \y in {4,2} \draw (11,\y) node[fill=blue!30!white] {$v_\y$};
  \foreach \y in {3,1} \draw (11,\y) node[fill=blue!30!white] {$v_\y$};
  \node[draw=none,fill=none] at (11,0) {$P_1$};
  \draw (12,1)--(12,4);
  \foreach \y in {4,2} \draw (12,\y) node[fill=blue!30!white] {$v_\y$};
  \foreach \y in {3,1} \draw (12,\y) node[fill=blue!30!white] {$v_\y$};
  \node[draw=none,fill=none] at (12,0) {$P_2$};
  \draw (13,1)--(13,4);
  \foreach \y in {4,2} \draw (13,\y) node[fill=red!20!white] {$v_\y$};
  \foreach \y in {3,1} \draw (13,\y) node[fill=red!50!white] {$v_\y$};
  \node[draw=none,fill=none] at (13,0) {$P_3$};
  \draw (14,1)--(14,4);
  \foreach \y in {4,2} \draw (14,\y) node[fill=green!25!white] {$v_\y$};
  \foreach \y in {3,1} \draw (14,\y) node[fill=green!60!white] {$v_\y$};
  \node[draw=none,fill=none] at (14,0) {$P_4$};
  \node[draw=none,fill=none] at (10.1,0.5) {(k)};
\end{tikzpicture} 

\begin{tikzpicture}[xscale=\xscal, yscale=\yscal]
  \tikzstyle{every node}=[draw, black, shape=circle, minimum size=3pt,inner sep=0.5pt, fill=gray]
  \tikzstyle{every path}=[draw, color=gray]
  \foreach \y in {4,3,2,1} {
    \draw (0,\y) node[label=left:$v_\y$] {} to[out=90,in=90] ++(0.33,0) to[out=270,in=270] ++(-0.33,0);
  } \draw (0,1)--(0,4);
  \node[draw=none,fill=none] at (0,0) {$P\in\ca H$};
  \tikzstyle{every path}=[draw, color=black, thick]\small
  \foreach \y in {4,3,2,1} \draw (3,\y)--++(5,0);
  \draw (3,1)--(3,4);
  \foreach \y in {4,2} \draw (3,\y) node[fill=blue!30!white] {$v_\y$};
  \foreach \y in {3,1} \draw (3,\y) node[fill=blue!30!white] {$v_\y$};
  \node[draw=none,fill=none] at (3,0) {$P_1$};
  \draw (4,1)--(4,4);
  \foreach \y in {4,2} \draw (4,\y) node[fill=blue!30!white] {$v_\y$};
  \foreach \y in {3,1} \draw (4,\y) node[fill=blue!30!white] {$v_\y$};
  \node[draw=none,fill=none] at (4,0) {$P_2$};
  \draw (5,1)--(5,4);
  \foreach \y in {4,2} \draw (5,\y) node[fill=blue!30!white] {$v_\y$};
  \foreach \y in {3,1} \draw (5,\y) node[fill=blue!30!white] {$v_\y$};
  \node[draw=none,fill=none] at (5,0) {$P_3$};
  \draw (6,1)--(6,4);
  \foreach \y in {4,2} \draw (6,\y) node[fill=blue!30!white] {$v_\y$};
  \foreach \y in {3,1} \draw (6,\y) node[fill=blue!30!white] {$v_\y$};
  \node[draw=none,fill=none] at (6,0) {$P_4$};
  \draw (7,1)--(7,4);
  \foreach \y in {4,2} \draw (7,\y) node[fill=blue!30!white] {$v_\y$};
  \foreach \y in {3,1} \draw (7,\y) node[fill=blue!30!white] {$v_\y$};
  \node[draw=none,fill=none] at (7,0) {$P_5$};
  \foreach \y in {4,3,2,1} \draw[semithick,densely dotted] (8,\y)--++(1,0);
  \draw (8,1)--(8,4);
  \foreach \y in {4,2} \draw (8,\y) node[fill=green!25!white] {$v_\y$};
  \foreach \y in {3,1} \draw (8,\y) node[fill=green!60!white] {$v_\y$};
  \node[draw=none,fill=none] at (8,0) {$P_6$};
  \draw (9,1)--(9,4);
  \foreach \y in {4,2} \draw (9,\y) node[fill=red!20!white] {$v_\y$};
  \foreach \y in {3,1} \draw (9,\y) node[fill=red!50!white] {$v_\y$};
  \node[draw=none,fill=none] at (9,0) {$P_7$};
  \foreach \y in {4,3,2,1} \draw[thick,loosely dotted] (9.5,\y)--++(1,0);
  \node[draw=none,fill=none] at (2.1,0.5) {(l)};
\end{tikzpicture} 
\caption{A $(P,5)$-expression making an $a\times b$ square grid, where $P$ is a $b$-vertex loop path:\newline
Left: The loop path $P\in\ca H$ with the parameter vertices $v_1,v_2,\ldots,v_b$ (here~$b=4$).
Right: Making the grid with a $(P,5)$-expression, left-to-right and top-to-bottom in order (a)--(l), as follows.\newline
(a) Create two vertices labelled $(1,v_1)$ (red) and $(5,v_2)$ (blue).
(b) Add edges from red to blue. (c)~Recolour $(5,v_2)$ to $(2,v_2)$ (light red) and add a new vertex labelled $(5,v_3)$ (blue).
(d) Add edges from light red to blue, then recolour blue and add a new vertex labelled $(5,v_4)$ (blue).
(e) Add edges from red to blue again, and then recolour blue to light red, finishing a copy $P_1$ of the path $P$.\newline
(f) Analogously create a copy $P_2$ of $P$, coloured alternately green and light green.
(g) Add edges from red to green -- note that this creates only the two ``horizontal'' edges because of~$P$.
(h) Analogously add edges from light red to light green.
(i) Recolour both red and light red to blue, and then add a copy $P_3$ of $P$ coloured red and light red.
(j) Add edges from green to red and from light green to light red, as before.
(k) Analogously recolour both green and light green to blue, and then add a copy $P_4$ of $P$ coloured green and light green.
(l) Continue this construction up to desired size.}
\label{fig:hcw-grid}
\end{figure}

To further briefly illustrate \Cref{def:hcw}, we add few more easy observations:
\begin{claim}\label{clm:basics}
\begin{enumerate}[a)]
\item\label{it:singlev} If $\ca H=\{K_1\}$, then $\hcw{\ca H}(G)<\infty$ if and only if $G$ has no edges.
If $\ca H=\{K_1^\circ\}$, where $K_1^\circ$ stands for a single vertex with a loop, then $\hcw{\ca H}(G)=\cw(G)<\infty$.
\item If $\ca H$ contains a loop graph with at least one loop, then $\hcw{\ca H}(G)\leq\cw(G)$.
\item\label{it:H2} If $\ca H=\{H\}$, then $\hcw{\ca H}(H_1)\leq2$ holds for every simple graph $H_1$
obtained from an induced subgraph of $H$ by removing loops.
\item\label{it:Hcompo} If $\ca H=\{H\}$ and $H$ is disconnected, then for every connected simple graph $G$ we have 
$\hcw{\ca H}(G)=\hcw{\{H_0\}}(G)$ for some connected component $H_0$ of~$H$.
\item\label{it:homom}
If $\ca H=\{K_2\}$ (no loops), then $\hcw{\ca H}(G)<\infty$ if and only if $G$ is a simple bipartite graph.
More generally, for any $\ca H$, we have $\hcw{\ca H}(G)<\infty$ for a  simple graph $G$, if and only if $G$ has a homomorphism into some $H\in\ca H$.
\item\label{it:nphk}
For any $k\geq3$ and $\ca H=\{K_k\}$, it is NP-hard to decide whether $\hcw{\ca H}(G)<\infty$.
\item\label{it:grids}
Let $\ca H$ be a family containing arbitrarily long reflexive paths.
If $G$ is any square grid, then $\hcw{\ca H}(G)\leq5$ while $\cw(G)$ is generally unbounded.
\end{enumerate}
\end{claim}

\begin{proof}
a) There is no edge in $K_1$, and so \Cref{def:hcw} cannot create an edge in~$G$.
On the other hand, $(K_1^\circ,\ell)$-expressions in \Cref{def:hcw} exactly coincide with traditional $\ell$-expressions of clique-width
(replacing every label $i$ with $(i,v)$ where $\{v\}=V(K_1^\circ)$).

b) We pick $v\in V(H)$ for $H\in\ca H$ such that $v$ has a loop in~$H$.
Then, in an ordinary $\cw(G)$-expression for $G$, we replace every label $i$ with $(i,v)$ to get an $(H,\cw(G))$-expression for $G$,
similarly to part a).

\ref{it:H2}) We simply make an $(H,2)$-expression for $H_1$ as follows;
in an arbitrary order $V(H_1)=\{v_1,\ldots,v_n\}$ of the vertices, for $k=1,\ldots,n$, we add a new vertex labelled $(2,v_k)$,
add edges between $1$ and $2$, and recolour $2$ to $1$. This creates exactly the non-loop edges of~$H_1$.

\ref{it:Hcompo}) 
This follows from the facts that the recolouring operation of \Cref{def:hcw} does not allow to change the initially assigned
parameter vertex of $H$, and hence every edge of $G$ created within an $(H,\ell)$-expression has a preimage edge in~$H$.
So, an expression creating connected $G$ may only use parameter vertices of a connected component of~$H$.

\ref{it:homom}) 
Considering the previous argument turned around, 
every edge created within an $(H,\ell)$-expression has a unique homomorphic image in~$H$ (possibly a loop).
In the opposite direction, for a homomorphism $h:G\to H\in\ca H$, we make an $(H,|V(G)|)$-expression starting with 
the disjoint union of vertices labelled $(i_x,h(x))$ for all $x\in V(G)$ where $i_x\not=i_y$ for $x\not=y$, and then
simply add the edges of $G$ one by one using the colours (i.e.,~$i_x$).

\ref{it:nphk})
By \ref{it:homom}), we have $\hcw{\ca H}(G)<\infty$ if and only if $G$ is $k$-colourable.

\ref{it:grids}) Let $G$ be an $a\times b$ grid, i.e.,~$|V(G)|=ab$.
We pick $P\in\ca H$ such that $P$ is a reflexive path of length $|V(P)|\geq b$, 
choose a consecutive subpath of $P$ on vertices $\{v_1,\ldots,v_b\}\subseteq V(P)$ in this order,
and construct an $(H,5)$-expression valued $G$ as follows.

Similarly to \ref{it:H2}), we define an $(H,3)$-(sub)expression creating a ``vertical'' copy $P_1$ of the path on $b$ vertices,
but now using three colours $1,2,5$ such that the resulting labels of $P_1$ are with alternating colours $1$ and $2$,
precisely as $(1,v_1)$, $(2,v_2)$, $(1,v_3)$, $(2,v_4),\ldots$.
We likewise create a copy $P_2$ of the same path using colours $3,4,5$ such that the result has labels
with alternating colours as $(3,v_1)$, $(4,v_2)$, $(3,v_3)$, $(4,v_4),\ldots$.
Then we make a disjoint union and add edges between colours $1,3$ and between $2,4$ --
this creates precisely the ``horizontal'' edges between the labels $(1,v_i)$ and $(3,v_i)$, and between $(2,v_{i+1})$ and $(4,v_{i+1})$,
for $i=1,3,\ldots$.
In a subsequent round, we recolour colours $1,2$ to $5$ (this concerns only~$P_1$), and continue an analogous process with
adding a path $P_3$ with alternating colours again $1$ and $2$, and adding the ``horizontal'' edges.
After $a-1$ rounds, we build the desired $a\times b$ square grid~$G$.

This construction is illustrated for $b=4$ in \Cref{fig:hcw-grid}.
\end{proof}

\section{Properties of $\ca H$-Clique-Width}
\label{sec:propehcw}

We first characterise the asymptotic difference between ordinary clique-width and $\ca H$-clique-width
for families of loop graphs $\ca H$.

We recall the concept of neighbourhood diversity by Lampis \cite{DBLP:journals/algorithmica/Lampis12}.
Two vertices $x,y$ of a simple graph $G$ are of the same {\em neighbourhood type} 
if and only if they have the same set of neighbours in $V(G)\setminus\{x,y\}$.
The {\em neighbourhood diversity} of $G$ is the least integer $d$ for which $V(G)$ can be partitioned into $d$ parts such that
every pair in the same part has the same neighbourhood type.

We actually use an adjusted version of this concept.
Two vertices $x,y$ of a simple loop graph $G$ are of the same {\em total neighbourhood type}
if and only if they have the same set of neighbours in~$V(G)$ when $x$ counts as a neighbour of $x$ if there is a loop on~$x$
(and likewise with~$y$).
A loop graph $G$ is of {\em total neighbourhood diversity} at most $d$ if $V(G)$ can be partitioned into $d$ parts such that
every pair in the same part has the same total neighbourhood~type.

The slight, but very important in our context, difference of these two notions in the presence of loops can be observed, e.g., on:
loopless cliques $K_n$ (neighbourhood diversity $1$ and total neighbourhood diversity~$n$)
vs.\ reflexive cliques $K_n^\circ$ (total neighbourhood diversity~$1$),
or loopless stars $K_{1,n}$ (both neighbourhood diversities equal~$2$) vs.\
reflexive stars $K_{1,n}^\circ$ (total neighbourhood diversity~$n$).

A loop graph class $\ca G$ is of {\em component-bounded} total neighbourhood diversity if there exists an integer
$d$ such that each connected component of every graph of $\ca G$ is of total neighbourhood diversity at most~$d$.

\begin{theorem}\label{thm:bddiversity}
Let $\ca H$ be a family of loop graphs.
There exists a function $f$ such that, $\cw(G)\leq f(\hcw{\ca H}(G))$ holds for all simple graphs $G$,
if and only if $\ca H$ is of component-bounded total neighbourhood diversity.
\end{theorem}

\begin{proof}
In the `$\Leftarrow$' direction, we may assume $G$ is connected since we will later take the maximum over connected components.
By \Cref{clm:basics}\,\ref{it:Hcompo}), $\hcw{\ca H}(G)=\hcw{\{H_0\}}(G)=\ell$ for a connected component $H_0$ of some~$H\in\ca H$.
The total neighbourhood diversity of $H_0$ is at most some constant~$d$, by the theorem assumption.
Then, in an $(H_0,\ell)$-expression for~$G$, we may equivalently replace the parameter vertices of $H_0$ by $d$ new colours,
giving a $d\ell$-expression for $G$. So, $\cw(G)\leq d\cdot\hcw{\ca H}(G)$.

A proof of the `$\Rightarrow$' direction is based on the following natural technical claim:
\begin{claim}[Ding et al.~{\cite[Corollary~2.4]{DBLP:journals/jct/DingOOV96}}]\label{clm:threematch}
For every $k$ there exists $m$ such that the following holds.
If $F$ is a bipartite simple graph with the bipartition $V(F)=A\cup B$, $|A|\geq m$ and the vertices
of $A$ have pairwise different neighbourhood types (in~$B$), then $F$ contains an induced subgraph
isomorphic to one of the following graphs on $2k$ vertices: a matching, a co-matching, or a half-graph.
\end{claim}

Having \Cref{clm:threematch} at hand, we continue as follows.
Assume that $\ca H$ is {\em not} of component-bounded total neighbourhood diversity.
Let $H\in\ca H$ (or a component thereof) be a connected loop graph of total neighbourhood diversity at least~$c_1$,
and $C\subseteq V(H)$ be vertices representing these $c_1$ total neighbourhood types.
By Ramsey's theorem, for sufficiently large $c_1$ we find a subset $C_1\subseteq C$, $|C_1|=2c_2-1$,
such that $C_1$ induces a clique or an independent set in~$H$,
and then we can select $C_2\subseteq C_1$, $|C_2|=c_2$ such that either all vertices of $C_2$ have loops, or none has.
We have got one of the two possibilities:
\begin{itemize}
\item $C_2$ is a reflexive independent set or a loopless clique in~$H$.
\item Or, all vertices of $C_2$ have the same total neighbourhood type in $C_2$ (empty or full~$C_2$), and so
they have pairwise different neighbourhood types in $D:=V(H)\setminus C_2$.
Consequently, we may apply \Cref{clm:threematch} to the bipartite subgraph ``between'' $C_2$ and $D$.
\end{itemize}
Regarding the second point, \Cref{clm:threematch} hence says that one of the three claimed kinds of subgraphs is bi-induced in~$H$.

Altogether, for every $k$ and sufficiently large $c_1$ depending on $k$,
we have connected $H\in\ca H$ containing one of the five mentioned substructures;
an induced reflexive independent set or an induced loopless clique on $k$ vertices, 
or a bi-induced matching, a bi-induced co-matching, or a bi-induced half-graph on $2k$ vertices.
In each of these five cases, we can construct a ``grid-like'' graph of bounded $\ca H$-clique-width
whose ordinary clique-width grows linearly with~$k$.
This is provided by subsequent \Cref{lem:highcwc}, in which one can easily check that its assumptions cover 
all five cases of $H\in\ca H$ listed in this proof (note that a reflexive independent set and a loopless clique arise with $A=B$ in the lemma).
\end{proof}

\begin{lemma}\label{lem:highcwc}
Let $k\geq3$ be an integer, and $H_1$ be a loop graph satisfying the following:
\begin{itemize}
\item $H_1$ is connected.
\item There exist sets $A,B\subseteq V(H_1)$, $|A|=|B|=k$, such that either $A=B$, or~$A\cap B=\emptyset$.
\item We can write $A=\{u_1,\ldots,u_k\}$ and $B=\{u'_1,\ldots,u'_k\}$ such that, for some of the following three conditions
on integers $C(i,j)\in\{$`\,$i<j$', `\,$i=j$', `\,$i\not=j$'$\}$ we have;
for all $(i,j)\in\{1,\ldots,k\}^2$, $\{u_i,u'_j\}\in E(H_1)$ if and only if $C(i,j)$ is false.
(Note that, if $A=B$, we assume~$u_i=u'_i$ and deal also with loops, and the condition `\,$i<j$' is impossible then.)
\end{itemize}
Then the class of graphs of $\{H_1\}$-clique-width at most~$3$ has ordinary clique-width of order~$\Omega(k)$.
\end{lemma}

\begin{proof}
We construct, via an $(H_1,3)$-expression, a graph $G_k$ of ordinary clique-width~$\Omega(k)$ as follows.

For $a\in\{1,\ldots,k\}$ we create, similarly as in \Cref{clm:basics}\,\ref{it:H2}), 
a loopless labelled copy $G'_a$ of the graph $H_1$, such that the following two properties hold:
every vertex $x\in V(G'_a)$ which is a copy of a vertex $v\in A$ if $a$ is odd (a copy of $v\in B$ if $a$ is even) 
has the label with colour $2$ and parameter vertex $v$, and all remaining vertices of $G'_a$ have labels with colour~$1$ and their parameter vertices are irrelevant.
The vertices of each $G'_a$ with colour $2$ are further called {\em active}.

We set $G_1:=G'_1$, and for $a=2,\ldots,k$ we do:
\begin{itemize}
\item In $G_{a-1}$ we recolour $2$ to $3$.
\item Then we make the disjoint union~$G_a:=G_{a-1}\,\dot\cup\,G'_a$, and add edges between colours $2$ and $3$.
\item In $G_{a}$ we recolour $3$ to $1$.
\end{itemize}

Altogether, the graph $G_k$ has $k\cdot|V(H_1)|$ vertices, contains $k$ disjoint copies $G'_a$ of $H_1$ and is connected.
For every $a\in\{2,\ldots,k\}$ the edges between $V(G'_{a-1})$ and $V(G'_{a})$ in $G_k$ induce (with their end vertices)
a subgraph $G''_a$, such that $G''_a$ is isomorphic to $H_1[A\cup B]$ if $A\not=B$, and $G''_a$ is isomorphic to either a matching or a co-matching otherwise.
There are no other edges in~$G_k$ than these described ones.
For clarity, we imagine the copy $G'_a$ of $H_1$ as the ``column~$a$'' of $G_k$, 
and the set of the active copies of the vertices $u_i$ and $u'_i$ of $A\cup B$ as the ``row~$i$'' of $G_k$ for $i\in\{1,\ldots,k\}$.
The non-active vertices are not considered a part of any row in this setting.

Now, assume we have an (ordinary) $\ell$-expression $\varphi$ valued~$G_k$ for some integer~$\ell$. 
We apply an argument which is folklore in this area.
There must exist a subexpression $\varphi_1$ of $\varphi$ making a subset of vertices $X\subseteq V(G_k)$
such that $\frac13|V(G_k)|\leq|X|\leq\frac23|V(G_k)|$.
Note that it will be irrelevant which of the edges of $G_k[X]$ our $\varphi_1$ makes.
Let $\bar X=V(G_k)\setminus X$.

For $a\in\{2,\ldots,k\}$, let $J\subseteq\{1,\ldots,k\}$ be the set of rows $j\in J$ in which the columns $a-1$ and $a$ differ with respect to $X$;
meaning that out of the two vertices in row $j$ and columns $a-1$ and $a$, exactly one belongs to $X$ and the other one to~$\bar X$.
Since the vertices of every two rows from $J$ have distinct adjacencies between columns $a-1$ and $a$, as given by the condition $C(i,j)$ of the lemma,
we necessarily have~$|J|\leq2\ell$.
In other words, the columns $a-1$ and $a$ differ with respect to $X$ in at most $2\ell$ rows.

Likewise, at most $\ell$ columns are such that they intersect both $X$ and $\bar X$.
This follows similarly since every column is a copy of connected $H_1$, and so it needs in $\varphi_1$ a special colour
for its (at least one) ``private'' edge from $X$ to $\bar X$.
The two latter conditions together are in a clear contradiction with $\frac13|V(G_k)|\leq|X|\leq\frac23|V(G_k)|$ if~$\ell\in o(k)$.
\end{proof}

Secondly, there is an interesting relation to established concepts in the case of parameter families $\ca H$ of bounded degrees,
and in particular of bounded bandwidth which includes the case of $\ca H$ being the family of loop paths.

\begin{theorem}\label{thm:localcw}
Let $\ca H$ be a family of loop graphs of maximum degree~$\Delta$.
Then the class of graphs of $\ca H$-clique-width at most $\ell$ is of bounded local clique-width in terms of $\Delta$ and~$\ell$.
Furthermore, if $\ca H$ is of bandwidth at most~$b$, 
then the class of graphs of $\ca H$-clique-width at most $\ell$ is of {\em linearly} bounded local clique-width in terms of $b$ and~$\ell$;
specifically, this bounding function is $(2b\ell\cdot r)$ for radius~$r$.
\end{theorem}

\begin{proof}
Let $H\in\ca H$ and $G$ be a graph that is a value of an $(H,\ell)$-expression $\varphi$.
Choose $x\in V(G)$, and assume a vertex $y\in V(G)$ at distance at most $r$ from $x$ in~$G$.
Let $v,w\in V(H)$ be the parameter vertices in $\varphi$ of $x$ and $y$, respectively.
As argued in \Cref{clm:basics}\,\ref{it:homom}), there is a homomorphism $G\to H$ taking a path
between $x$ and $y$ into a walk between $v$ and $w$ in $H$, and so the distance from $v$ to $w$ in $H$ is at most~$r$.
Since $\Delta(H)\leq\Delta$, the $r$-neighbourhood of $v$ in $H$ has at most $(\Delta+1)^r$ vertices,
and hence $\varphi$ restricted to the $r$-neighbourhood of $x$ in $G$ uses only at most $(\Delta+1)^r$ parameter vertices
which can be replaced in $\varphi$ by unique colours.
This way we obtain an (ordinary) $\ell\cdot(\Delta+1)^r$-expression whose value is the $r$-neighbourhood of $x$ in $G$.
We can thus set $f(r):=\ell\cdot(\Delta+1)^r$ (independently of $H\in\ca H$ and $G$) 
to certify bounded local clique-width of every~$G$ such that $\hcw{\ca H}(G)\leq\ell$.

Furthermore, if $H$ is of bandwidth $b$, then the $r$-neighbourhood of any vertex $v$ in $H$ has at most $2br$ vertices,
and so we can choose $f(r):=2b\ell\cdot r$ (independently of $H\in\ca H$ and $G$) as the local clique-width bounding function.
\end{proof}

A similar structural relation of $\ca H$-clique-width to the parameter twin-width is stated later in
\Cref{cor:Hcw-tww}, as a consequence of a product-structure-like characterisation.

From \Cref{thm:localcw} we, for instance, immediately get tractability of FO model checking,
which is FPT for all classes of bounded local clique-width -- this well-known fact follows by a combination
of the ideas of Frick and Grohe~\cite{DBLP:journals/jacm/FrickG01} and of
Dawar, Grohe and Kreutzer~\cite{DBLP:conf/lics/DawarGK07}:
\begin{corollary}\label{cor:localcw}
For every family $\ca H$ of loop graphs, the FO model checking problem of a graph $G$ is in FPT 
when parametrised by the formula, the maximum degree of $\ca H$ and the $\ca H$-clique-width of~$G$.
\qed\end{corollary}

Furthermore, it may be interesting to ask to which extent \Cref{thm:localcw} can be reversed.
This cannot be done straightforwardly since there are families $\ca H$ of unbounded degrees, such that
classes of bounded $\ca H$-clique-width not only have bounded local clique-width, but even bounded ordinary clique-width.
One example is $\ca H_1$ the class of all reflexive cliques by \Cref{thm:bddiversity}.
On the other hand, e.g., for $\ca H_2$ being the class of all reflexive stars, there are graphs whose $\ca H_2$-clique-width
is bounded by a constant, and they contain arbitrarily large induced grids and a universal vertex adjacent to everything
(a construction similar to \Cref{clm:basics}\,\ref{it:grids})\,).
Such graphs hence have unbounded local clique-width.

\section{Approaching Hereditary Graph Product Structure}
\label{sec:induprod}

We now directly relate the notion of $\ca H$-clique-width to the hereditary product structure, and use this view to derive other simple properties.

From this section onward we restrict our attention to families $\ca H$ formed by {\em reflexive} loop graphs (i.e., all vertices have loops
in the graphs of $\ca H$), which makes the most natural sense with respect to the strong-product structure studied.

\begin{theorem}\label{thm:hcwproduct}
Let $\ca H$ be a family of reflexive loop graphs,
and $\ca H'$ be the family of simple graphs obtained from the graphs of $\ca H$ by removing all loops.
For every integer $\ell\geq2$, the following holds.
A simple graph $G$ is of $\ca H$-clique-width at most~$\ell$, if and only if $G$ is isomorphic to an induced subgraph
of the strong product $H'\boxtimes M$ where $H'\in\ca H'$ and $M$ is a simple graph of clique-width at~most~$\ell$.
\end{theorem}

\begin{proof}
In the `$\Leftarrow$' direction, it is enough to show that $\hcw{\ca H}(G)\leq\ell$ for~$G:=H'\boxtimes M$.
Let $\varphi$ be an $\ell$-expression of the graph $M$, and let $H^\circ$ be obtained from $H'$ by adding loops to all vertices.
We are going to transform $\varphi$ into an $(H^\circ,\ell)$-expression as follows.
First, for each $x\in V(M)$ we independently construct a copy $H'_x$ of $H'$, using only $2\leq\ell$ colours by \Cref{clm:basics}\,\ref{it:H2}).
That is, the parameter vertex of every $v_x\in V(H'_x)$ is the preimage $v\in V(H')$ of~$v_x$.
Then, at every moment the expression $\varphi$ introduces a new vertex $y\in V(M)$ of colour~$i$,
we take (substitute) the copy $H'_y$ and recolour it to~$i$.
The remaining operations (union, recolouring, and adding edges) stay in place in $\varphi$, but are now applied according to \Cref{def:hcw}.

We claim that the value $G$ of the resulting transformed $(H^\circ,\ell)$-expression $\sigma$ is~$H'\boxtimes M$.
Indeed, the vertex set is $V(G)=V(H')\times V(M)$, and for each $m\in V(M)$ the subgraph induced on $V(H')\times\{m\}$ is a copy of~$H'$.
For any $[v_1,m_1],[v_2,m_2]\in V(G)$ and $m_1\not=m_2$; 
if $\{[v_1,m_1],[v_2,m_2]\}\in E(G)$, then $v_1v_2\in E(H^\circ)$ by \Cref{def:hcw}, and $m_1m_2\in E(M)$ by the definition of~$\sigma$.
Hence $\{[v_1,m_1],[v_2,m_2]\}\in E(H'\boxtimes M)$.
Conversely, if $\{[v_1,m_1],[v_2,m_2]\}\in E(H'\boxtimes M)$, then, by the definition of $\boxtimes$,
$v_1v_2\in E(H')$ or $v_1=v_2$, meaning $v_1v_2\in E(H^\circ)$, and $m_1m_2\in E(M)$.
So, the edge $\{[v_1,m_1],[v_2,m_2]\}$ has been created by~$\sigma$.

\smallskip
In the `$\Rightarrow$' direction, let $\sigma$ be an $(H^\circ,\ell)$-expression valued~$G$, for some $H^\circ\in\ca H$.
Let $H'\in\ca H'$ be the simple graph of $H^\circ$.
We are going to construct an $\ell$-expression $\varphi$ valued $M$ on the vertex $V(M)=V(G)$, such that $G\subseteq_i H'\boxtimes M$.
The expression $\varphi$ simply discards parameter vertices (cf.~\Cref{def:hcw}) from the labels in $\sigma$.
Hence, we clearly get~$M\supseteq G$.
To prove that $G\subseteq_i H'\boxtimes M$, consider any vertices $x\not=y\in V(G)$ labelled $(i,v)$ and $(j,w)$ by~$\sigma$.
Here, $v$ and $w$ are uniquely determined by~$\sigma$, and indefinite $i$ and $j$ are irrelevant for our argument.
We claim that the vertices $x$ and $y$ as of $G$ can be represented by $[v,x]$ and $[w,y]$ of the product~$H'\boxtimes M$.
If $xy\in E(G)$, then $xy\in E(M)$ by previous~$M\supseteq G$, and $vw\in E(H^\circ)$ by \Cref{def:hcw}.
Consequently, $\{[v,x],[w,y]\}\in E(H'\boxtimes M)$ by~$\boxtimes$.
On the other hand, if $\{[v,x],[w,y]\}\in E(H'\boxtimes M)$, then $vw\in E(H')$ or $v=w$, and so $vw\in E(H^\circ)$.
Moreover, $xy\in E(M)$ since $x\not=y$, and so $xy\in E(G)$ since the (original) $(H^\circ,\ell)$-expression $\sigma$ creates the edge~$xy$
by~\Cref{def:hcw}.
\end{proof}

\Cref{thm:hcwproduct} can be used also to bound the twin-width of graphs of bounded $\ca H$-clique-width.
To show this, we first prove the following ad-hoc upper bound.

\begin{proposition}\label{pro:Pcw-tww}
Let $P$ be a reflexive path and $G$ a simple graph.
Then the twin-width of $G$ is at most $5\cdot(\hcw{\{P\}}(G))-2$.
Consequently, denoting by $\ca P^\circ$ the class of all reflexive paths,
the twin-width of any simple graph $G$ is at most $5\cdot(\hcw{\ca P^\circ\!}(G))-2$.
\end{proposition}
\begin{proof}
Let $G$ be the value of a $(P,\ell)$-expression $\varphi$, where $\ell=\hcw{\{P\}}(G)$.
When constructing a contraction sequence for $G$, we proceed recursively (bottom-up) along 
the expression tree of~$\varphi$; processing only the union and recolouring nodes, and at each node
contracting together all vertices of the same label.

Consider a situation at a node with a subexpression $\varphi_{0}$ of $\varphi$, 
where $X_0\subseteq V(G)$ is the vertex set generated by $\varphi_0$,
and let $x^0_{(i,v)}$ denote the vertex resulting from the contractions of all vertices of $X_0$
that are of label $(i,v)$ by $\varphi_0$.
The core observation is that every vertex of $V(G)\setminus X_0$ has the same adjacency to all vertices forming
$x^0_{(i,v)}$ by \Cref{def:hcw}, and so the only possible red neighbours of $x^0_{(i,v)}$
in a contraction of $G$ are those that stem from~$X_0$.

The only possible neighbours of $x^0_{(i,v)}$ in the described contraction of the induced subgraph $G[X_0]$
are $x^0_{(j,w)}$ where $j\in\{1,\ldots,\ell\}$ and $vw\in E(P)$ --
altogether at most $3\ell-1$ choices of potential red neighbours of $x^0_{(i,v)}$ in the contraction of $G[X_0]$.
If a recolouring operation $i$ to $j$ is encountered after the node of $\varphi_0$, we simply contract each
former $x^0_{(i,v)}$ with $x^0_{(j,v)}$ over all $v\in X_0$, not increasing the previous bound on the red degree.

Consider now a union node making $X_2:=X_0\dot\cup X_1$, where $X_1$ has been generated by a sibling
subexpression $\varphi_1$ of $\varphi$, and let $x^1_{(i,v)}$ analogously denote the vertices resulting from contractions of~$X_1$.
Let the vertices of $P$ be $V(P)=(v_1,\ldots,v_a)$ in the natural order along the path.
For $k=1,\ldots,a$, and subsequently for $i=1,\ldots,\ell$, we make $x^2_{(i,v_k)}$ by contracting $x^0_{(i,v_k)}$ with $x^1_{(i,v_k)}$.
Considering the corresponding successive contractions of the induced subgraph $G[X_2]$,
the only possible red neighbours of $x^2_{(i,v_k)}$ are the $\ell$ vertices $x^2_{(i',v_{k-1})}$,
the up to $2(\ell-1)$ vertices $x^2_{(j,v_k)}$ for $j<i$, or $x^0_{(j,v_k)}$, $x^1_{(j,v_k)}$ for $j>i$,
and the $2\ell$ vertices $x^0_{(i'',v_{k+1})}$, $x^1_{(i'',v_{k+1})}$.
The maximum possible encountered red degree is thus $\ell+2(\ell-1)+2\ell=5\ell-2$.
\end{proof}

\begin{corollary}\label{cor:Hcw-tww}
Let $\ca H$ be a family of reflexive loop graphs of maximum degree $\Delta$ and twin-width at most~$t$.
Then the twin-width of any simple graph $G$ is at most $\ca O(t+\Delta\cdot{\hcw{\ca H}(G)})$.
\end{corollary}
\begin{proof}
By \Cref{thm:hcwproduct}, we have $G\subseteq_i H'\boxtimes M$,
where $H'$ is of maximum degree at most $\Delta$ and twin-width at most~$t$ and $M$ is of clique-width at most~$\ell:=\hcw{\ca H}(G)$.
Since twin-width is monotone under taking induced subgraphs, it is enough to bound it for the graph $H'\boxtimes M$.

By \Cref{pro:Pcw-tww} applied to $P$ being a single reflexive vertex, and \Cref{clm:basics}\,\ref{it:singlev}),
$M$ is of twin-width at most $k:=5\ell-2$.
Then, by Bonnet et al.~\cite{DBLP:conf/soda/BonnetGKTW21} (bounding twin-width of a strong product), 
$H'\boxtimes M$ is of twin-width at most $\max\big\{t+\Delta,\,k(\Delta+1)+2\Delta\big\}=\ca O(t+\Delta\ell)$.
\end{proof}

Notice that, for constant $\Delta$, the bound $\ca O(t+\Delta\cdot{\hcw{\ca H}(G)})$ in \Cref{cor:Hcw-tww} is asymptotically best possible;
a linear dependence on $t$ (the maximum twin-width in~$\ca H$) is necessary due to \Cref{clm:basics}\,\ref{it:H2}),
and a linear dependence on $\hcw{\ca H}(G)$ is, on the other hand, required already by the subcase of ordinary clique-width.
It is not clear whether the linear dependence on $\Delta$ in the bound of \Cref{cor:Hcw-tww} is really necessary,
however, the next construction (\Cref{prop:squarestar}) shows that the bound has to grow with~$\Delta$, the maximum degree in~$\ca H$, as $\Omega(\sqrt[8]\Delta)$.

In a detail, let $S_n$ denote the star with $n$ leaves, let $S^\circ_n$ denote the reflexive star with $n$ leaves, and $\ca S^\circ=\{S^\circ_n:n\in\mathbb N\}$.
So, for every $n$, $\hcw{\ca S^\circ}(S_n\boxtimes S_n)\leq \cw(S_n)=2$ using \Cref{thm:hcwproduct} and the twin-width of $S_n$ is trivially~$0$.
We hence get that the bound in \Cref{cor:Hcw-tww} must grow with~$n=\Delta(S_n)$ if the twin-width of $S_n\boxtimes S_n$ grows with~$n$ as in \Cref{prop:squarestar}.
Recall that the {\em$k$-subdivision} of a graph $G$ is obtained by replacing every edge of $G$ with a path of length $k+1$ (i.e., a path with $k$ internal vertices).

\begin{proposition}\label{prop:squarestar}
For every graph $G$ there exists $n\in\mathbb N$ such that a graph isomorphic to the $3$-subdivision of $G$ 
is a bi-induced subgraph of $S_n\boxtimes S_n$.
Consequently, the twin-width of $S_n\boxtimes S_n$ is at least $\Omega(\sqrt[8]n)$.
\end{proposition}

\begin{proof}
We choose $n:=\max\{|V(G)|,|E(G)|\}$ (actually, considering the product \mbox{$S_{|V(G)|}\boxtimes S_{|E(G)|}$} would be enough).
Let $V(S_n)=\{c,l_1,\ldots,l_n\}$ where $c$ is the central vertex.
Let~$G_3$ denote the $3$-subdivision of $G$ (which is a bipartite graph), and $A\cup B=V(G_3)$ be the bipartition of $G_3$ such that~$A\supseteq V(G)$.
We are going to find a graph $G'_3\subseteq_i S_n\boxtimes S_n$ and a bipartition $V(G'_3)=A'\cup B'$
such that $G'_3$ restricted to the edges with at most one end in $A'$ is isomorphic to $G_3$.

Let, for simplicity, $V(G)=\{1,2,\ldots,m\}$ where~$m=|V(G)|$, and $E(G)=\{e_1,e_2,\ldots,e_{m'}\}$ where~$m'=|E(G)|$.
We choose $A'=A'_1\cup A'_2\subseteq V(S_n\boxtimes S_n)$ where $A'_1=\{[l_1,c],\ldots,[l_m,c]\}$ and $A'_2=\{[c,l_1],\ldots,[c,l_{m'}]\}$.
The intention is that every vertex $u\in V(G)$ corresponds to $[l_u,c]\in A'_1$, and every edge $e_j\in E(G)$ corresponds to $[c,l_j]\in A'_2$ which 
is thought of as the ``middle vertex'' subdividing~$e_j$ in~$G_3$.
Then we choose $B'=\{[l_u,l_k],[l_v,l_k]: 1\leq k\leq m',e_k=uv\}\subseteq V(S_n\boxtimes S_n)$.

Observe that there are no edges of $S_n\boxtimes S_n$ with both ends in $B'$ since the leaves of $S_n$ form an independent set.
A vertex $[l_i,c]\in A'_1$ is adjacent to $[l_u,l_k]\in B'$ in $S_n\boxtimes S_n$ if and only if $i=u$,
and $[c,l_j]\in A'_2$ is adjacent to $[l_u,l_k]\in B'$ if and only if $j=k$ and~$u\in e_k$.
Consequently, every vertex $[c,l_k]\in A'_2$ has precisely two neighbours in $B'$ -- the vertices $[l_u,l_k]$ and $[l_v,l_k]$ where~$uv=e_k$,
and $[l_u,l_k]$ in turn has a unique neighbour in $A'_1$ -- the vertex $[l_u,c]$.
Altogether, the subgraph $G'_3$ of $S_n\boxtimes S_n$ induced by $A'\cup B'$, minus the edges of $G'_3$ with both ends in~$A'$, is isomorphic to $G_3$.

For the second part of \Cref{prop:squarestar}, we employ the following result of \cite{DBLP:conf/soda/BonnetGKTW21};
the $k$-subdivision $K_s^{(k)}$ of the clique $K_s$ has twin-width at least $\Omega(\!\sqrt[k+1]s)$.
In our case, $k=3$ and $n=\Theta(s^2)$.
For $G:=K_s$, consider the graph $G'_3$ obtained above as an induced subgraph of $S_n\boxtimes S_n$, 
which in turn contains bi-induced~$G_3\simeq K_s^{(3)}$.
Obviously, the twin-width of $S_n\boxtimes S_n$ is at least as high as that of our $G'_3$.
On the other hand, by aforementioned \cite{DBLP:conf/soda/BonnetGKTW21}, the twin width of $G_3\simeq K_s^{(3)}$ is $\Omega(\sqrt[8]n)$,
which leaves ``bridging the gap from $G'_3$ to $G_3$'' as the remaining task to be done.

Let $t$ be the twin-width of $G'_3$ and $\sigma$ be a corresponding contraction sequence of $G'_3$ of red degree at most~$t$.
We apply the same sequence $\sigma$ to $G_3$ with the following little modification:
for each vertex $v$ of the sequence $\sigma$ of $G'_3$, we keep in the sequence of $G_3$ two (at most) copies $v^A,v^B$ of~$v$ 
-- one collecting the vertices contracted into $v$ which originate from $A$, and one collecting those originating from~$B$.
Since $A$ and $B$ are independent sets of $G_3$, there are no red edges within $A$ and none within $B$ in the contraction sequence of $G_3$.
Each possible red edge of the form $v^Aw^B$ along the modified contraction sequence of $G_3$ either corresponds to a red edge $vw$
of the sequence of $G'_3$, or it is an edge $v^Av^B$.
Consequently, the twin-width of $G_3$ is at most~$t+1$, and so the twin-width of $G'_3$ is of order $\Omega(t)=\Omega(\sqrt[8]n)$.
\end{proof}

\section{Traditional Product Structure the Induced Way}
\label{sec:indutrad}

In regard of the Planar graph product structure theorem, as introduced in \Cref{sec:prelim} (\Cref{thm:origprod}), we are especially interested
in $\ca H$-clique-width for $\ca H=\ca P^\circ$ where $\ca P^\circ$ is the {\em class of reflexive paths}.
We get the following as another immediate consequence of \Cref{thm:hcwproduct}:

\begin{corollary}
For every integer $\ell\geq2$, the following holds.
A graph $G$ is isomorphic to an induced subgraph of the strong product $P\boxtimes M$ where $P$ is a path 
and $M$ is a simple graph of clique-width at most~$\ell$, if and only if\linebreak \mbox{$\hcw{\ca P^\circ\!}(G)\leq\ell$}.
\qed\end{corollary}
\medskip

There is, however, a more direct connection between our concept and the original Planar graph product structure theorem,
and this connection extends to all product-structure-like results that can be formulated within a strong product
of a factor of bounded degree and a factor of bounded tree-width.
This finding, formulated in \Cref{thm:fromprods}, constitutes the main new contribution of the paper.

Recall that the {\em square} $Q^2$ of a graph $Q$ is the graph on the same vertex set $V(Q^2)=V(Q)$ such that the edges of $Q^2$ are
formed by the pairs of vertices of $Q$ at distance $1$ or $2$ apart.

\begin{theorem}\label{thm:fromprods}
Let $Q$ be a simple graph of maximum degree $\Delta\geq2$ and $M$ be a simple graph of tree-width~$k$.
Assume that a graph $G$ is a subgraph (not necessarily induced) of the strong product $Q\boxtimes M$, that is, $G\subseteq Q\boxtimes M$.
Then:
\begin{enumerate}[a)]\parskip3pt\par
\item 
There exists a graph $M_1$ of clique-width $\ell\leq(\Delta^2+2)\cdot\Delta^{2(\Delta+1)(k+1)}$ such that $G$ is isomorphic to 
an {\em induced} subgraph of the strong product~$Q\boxtimes M_1$; $\>G\subseteq_i Q\boxtimes M_1$.
\\In particular, if $\ca H$ is a graph class containing the reflexive closure of $Q$, then
$\hcw{\ca H}(G)\leq\ell$.
\item 
There exists a graph $M_2$ of tree-width $\ell'\leq(k+1)(\Delta^2+1)\cdot\Delta^{2(\Delta+1)(k+1)}$ such that $G$ is isomorphic to 
an {\em induced} subgraph of the strong product~$Q\boxtimes M_2$; $\>G\subseteq_i Q\boxtimes M_2$.
\end{enumerate}

Furthermore, if the square of the graph $Q$ is $d$-colourable, then the graph $M_1$ can be chosen to have clique-width
$\ell\leq(d+1)\left(\min((d-1)^{(\Delta+1)},\,2^d)\right)^{k+1}$,
and the graph $M_2$ can be chosen of tree-width $\ell'\leq(k+1)d\left(\min((d-1)^{(\Delta+1)},\,2^d)\right)^{k+1}$.
\end{theorem}

\begin{proof}
We start with proving Part a) of the statement -- constructing the graph~$M_1$ of bounded clique-width.
Although, we remark that we could as well jump straight into a proof of Part b), and then conclude with an implied bound (albeit weaker) on the clique-width of~$M_2$.
We choose our gradual approach to the proofs of a) and b) also because we believe that it is more accessible for the readers.

By \Cref{thm:hcwproduct}, it suffices to construct a $(Q^\circ\!,\ell)$-expression $\varphi$ valued $G$, where $Q^\circ$ denotes the reflexive closure of~$Q$.

By a standard argument in this area, there exists a rooted tree decomposition $(T,\ca X)$ of width $k$ of the graph $M$, 
where $\ca X=\{X_t:t\in V(T)\}$, such that every node of $T$ has at most {\em two children}.
For a node $t\in V(T)$, let $X^+_t\subseteq V(M)$ denote the union of $X_s$ where $s$ ranges over $t$ and all descendants of~$t$.
Let $t^\uparrow$ denote the~parent node of $t$ in~$T$, and let $Y_t=X^+_t\setminus X_{t^\uparrow}$ denote the vertices of $M$ which
occur {\em only} in the bags of $t$ and its descendants.
For the root $r$ of $T$, let specially $Y_r=X^+_r=V(M)$.
Observe that all neighbours of a vertex $m\in Y_t$ in $V(M)\setminus Y_t$ must belong to the set $X_t\setminus Y_t$,
by the interpolation property of a tree decomposition.
We illustrate this important notation in \Cref{fig:treedecoxx}.

\begin{figure}[tb]
$$
\begin{tikzpicture}[scale=0.42]\small
  \tikzstyle{every node}=[draw, black, shape=circle, minimum size=3pt,inner sep=0pt, fill=black]
  \tikzstyle{every path}=[draw, color=black, thick]
  \node[label=left:$x$] at (-1.3,-1) (a) {};
  \node[label=left:$y$] at  (1.3,-1) (b) {};
  \node[label=left:$t$] at  (0,1) (c) {};
  \node[label=left:$t^\uparrow$] at  (0.5,3.9) (d) {};
  \draw (a)--(c)--(b) (c)--(d);
  \draw[dotted] (d)-- +(1.5,2.5)-- +(3,0);
  \draw (0,-3.2) node[draw=white, fill=white] {};
\end{tikzpicture} 
\qquad
\begin{tikzpicture}[scale=0.42]\small
  \tikzstyle{every node}=[draw, black, shape=circle, minimum size=3pt,inner sep=0pt, fill=none]
  \draw[rotate=30] (-1.3,0) ellipse (30mm and 20mm);
  \draw[rotate=-30] (1.3,0) ellipse (30mm and 20mm);
  \draw[rotate=0, thick] (0,1) ellipse (20mm and 33mm);
  \draw[rotate=-20] (0.25,3.9) ellipse (18mm and 31mm);
  \draw[dotted] (3,6) ellipse (22mm and 24mm);
  \draw[dotted,rotate=20] (6.2,1.5) ellipse (18mm and 31mm);
  \tikzstyle{every node}=[draw=none, black, shape=circle, minimum size=3pt,inner sep=0pt, fill=none]
  \node at (-3,-1) {$X_x$};
  \node at (3,-1) {$X_y$};
  \node at (-2.5,2.5) {$X_t$};
  \node at (-0.3,5.5) {$X_{t^\uparrow}$};
\end{tikzpicture} 
\quad
\begin{tikzpicture}[scale=0.42]\small
  \tikzstyle{every node}=[draw, black, shape=circle, minimum size=3pt,inner sep=0pt, fill=none]
  \draw[rotate=0, fill=red!30!white] (0,1) ellipse (20mm and 33mm);
  \draw[rotate=30, fill=green!30!white] (-1.3,0) ellipse (30mm and 20mm);
  \draw[rotate=-30, fill=green!30!white] (1.3,0) ellipse (30mm and 20mm);
  \draw[fill=red!30!white,draw=none] (-0.3,1.6) to[bend right=67] (2,1.1) to[bend right=27] cycle;
  \draw[fill=red!30!white,draw=none] (0,1.1) to[bend right=57] (1.6,0.9) to[bend right=27] cycle;
  \draw[fill=green!30!white,draw=none] (-2,1.15) to[bend left=12] (-0.3,1.65) to[bend left=16] (-0.06,4.33) to[bend right=27] (-1.6,3) to[bend right=12] cycle;
  \draw[rotate=0, thick] (0,1) ellipse (20mm and 33mm);
  \draw[rotate=30] (-1.3,0) ellipse (30mm and 20mm);
  \draw[rotate=-30] (1.3,0) ellipse (30mm and 20mm);
  \draw[rotate=-20] (0.25,3.9) ellipse (18mm and 31mm);
  \tikzstyle{every node}=[draw=none, black, shape=circle, minimum size=3pt,inner sep=0pt, fill=none]
  \node at (0,-0.7) {$Y_t$};
  \node at (0.7,2.3) {$X_t\!\setminus\!Y_t$};
  \node at (-2.5,2.5) {$X_t$};
  \node at (-0.3,5.5) {$X_{t^\uparrow}$};
\end{tikzpicture} 
$$
\caption{A fragment of a rooted tree decomposition $(T,\ca X)$ of a graph, showing a node $t$ with its parent $t^\uparrow$ and children $x,y$,
	and the corresponding bags $X_t$, $X_{t^\uparrow}$, $X_x$,~$X_y$ as subsets of the vertex set of the graph.
	On the right, we outline the sets $Y_t$ (green) -- all graph vertices occurring only in the bags of $t$ and its descendants,
	and $X_t\setminus Y_t$ (red) -- a separator between the set $Y_t$ and the rest of the graph.}
\label{fig:treedecoxx}
\end{figure}

Analogously to the treatment in the proof of \Cref{thm:hcwproduct}, we refer to the vertices of $G\subseteq Q\boxtimes M$
as to the pairs $[q,m]\in V(G)$ where $q\in V(Q)$ and $m\in V(M)$ in the natural correspon\-dence.
When constructing the expression $\varphi$ for the graph $G$, on a high level, we follow bottom-up the tree~$T$;
at a node $t\in V(T)$, we will construct an expression $\varphi^t$ whose value is precisely the subgraph $G^t$ of $G$ induced on 
the vertex set $V(G^t):=(V(Q)\times Y_t)\cap V(G)$.
By the previous, all neighbours of $V(G^t)$ in the rest of $G$ belong to the set $W_t:=\big(V(Q)\times(X_t\setminus Y_t)\big)\cap V(G)$, 
where $X_t\setminus Y_t$ is bounded in size and, furthermore, 
each vertex $[q,m]\in V(G^t)$ has neighbours $[q',m']\in V(G)$ of at most $\Delta+1$ distinct values of~$q'$.

It will thus be enough to encode, in the colour of each vertex $x=[q,m]\in V(G^t)$ within~$\varphi^t$, information about which vertices
of $W_t$ are actual neighbours of $x$ in $G$; moreover, these colours can be ``recycled'' such that we bound the total number of used ones in terms of~$\Delta$.
This way we will prove that the number of colours in our expression for $G$ would be bounded from above by the claimed expression.

Still on a high level, our construction of the $\varphi$ expression valued $G$ will look as follows.
\begin{enumerate}[{(A)}]
\item \label{it:firstVm}
We first create, in a trivial way, for each $m\in V(M)$ an expression $\nu_m$ for the edge-less vertex set $V_m:=\big(V(Q)\times\{m\}\big)\cap V(G)$, 
such that the labels in $V_m$ consist of the parameter vertices from $V(Q)$ in the natural correspondence, 
and of the initial colours described in further \Cref{sdef:colorx}.
\item\label{it:nextPhi}
We proceed bottom-up along the tree decomposition $(T,\ca X)$ of~$M$, starting with empty expressions for the (nonexisting) descendants of the leaves of $T$.
At each node $t\in V(T)$, we start the expression $\varphi^t$ with a disjoint union of the subexpressions of the children of~$t$.
Then, recalling that $Y_t$ are the vertices of $M$ not occurring in any ancestor bag of $X_t$,
we iteratively add the expressions $\nu_m$ for each $m\in Y_t\cap X_t$,
followed by adding all edges of $G^t$ incident to~$V_m\subseteq V(G^t)$, including those within $V_m$, based on previously assigned colours.
We finish $\varphi^t$ with recolouring all vertices of $V(G^t)$ to match the colours expected by further \Cref{sdef:colorx} at the parent $t^\uparrow$ of~$t$ (unless~$t=r$).
\end{enumerate}

\medskip

The crucial point of the whole proof is to define the colours used for the vertices in the expression $\varphi^t$ over $t\in V(T)$.
Let $S=\{1,2,\ldots,\Delta^2+1\}$, and let $s:V(Q)\to S$ be a proper colouring of the square graph~$Q^2$
-- such $s$ always exists since the maximum degree of $Q^2$ is at most~$\Delta^2$.
Furthermore, let $p:V(M)\to\{0,1,\ldots,k\}$ be a function such that $p$ is injective on each of the bags $X_t$ over $t\in V(T)$
-- such $p$ is easily constructed along a tree decomposition of width $k$ in the root-to-leaves order.
In fact, $p$ can alternatively be seen as a monotone cop search strategy on the tree decomposition $(T,\ca X)$ of~$M$.

\begin{subdefinition}\label{sdef:colorx}
Let $t\in V(T)$ and $x=[q,m]\in V(G^t)$ where $q\in V(Q)$ and, by the~definition of $G^t$, $m\in Y_t$.
The {\em base colour} of $x$ at the node~$t$ is the tuple $c^t_x=(b_0,b_1,\ldots,b_k)$ such that
\begin{itemize}
\item $b_j\in\{0,1\}^{S}$ for $0\leq j\leq k$, conveniently viewed as a function $b_j:S\to\{0,1\}$;
\item for every $j=p(m')$ where $m'\in X_t$, explicitly including $m'=m$ if $m\in X_t$, let
\begin{equation}\label{eq:bjdefi}
\mbox{$b_j(\alpha)=1$ iff $\{x,[q',m']\}\in E(G^t)$ for some $q'\in V(Q)$ such that $s(q')=\alpha$}
\end{equation}
(note that, as we will argue below, there is a unique choice of $m'$ for each such $j$, and a unique possible choice of such $q'$ for $\alpha\in S$);
\item for every $j\in\{0,1,\ldots,k\}\setminus p(X_{t})$, let $b_j(\alpha)=0$ for all $\alpha\in S$.
\end{itemize}
\smallskip
The actual colours of the vertices in $\varphi^t$ will be of the form $(\alpha,c)$ where $\alpha\in S\cup\{\perp\}$ and
$c\in\big(\{0,1\}^{S}\big)^{k+1}$ is a base colour as defined above. Namely,
\begin{itemize}
\item the pair $(\perp,c^t_x)$ is called the {\em running colour} of the vertex $x$ at node~$t$, and
\item if $m\in X_t\cap Y_t$, moreover, the pair $(s(q),c^t_x)$ is called the {\em initial colour} of the vertex $x$
(observe that there is a unique $t\in V(T)$ such that $m\in X_t\cap Y_t$, and hence no explicit reference to $t$ is necessary here).
\end{itemize}
\end{subdefinition}

The purpose of \Cref{sdef:colorx} will be more clear from the following claims.
Recall that $Q^\circ$ denotes the reflexive closure of the graph~$Q$.
\begin{claim}\label{clm:coledge}
Let $t\in V(T)$, and assume $x=[q,m]\in V(G^t)$ and $x'=[q',m']\in V(G^t)$ such that $m\in Y_t$ and $m'\in X_t\cap Y_t$ (possibly $m=m'$).
Let $c^t_x=(b_0,b_1,\ldots,b_k)$ be the base colour of $x$ at~$t$.
Then $xx'\in E(G^t)$, if and only if $qq'\in E(Q^\circ)$ and $b_{p(m')}(s(q'))=1$.
\end{claim}
\begin{proof}[Subproof.]
In one direction, if $xx'\in E(G^t)$, then $qq'\in E(Q^\circ)$ by the definition of $\boxtimes$,
and \Cref{sdef:colorx} states in \eqref{eq:bjdefi} that $b_{j}(\alpha)=1$ for $j=p(m')$ and $\alpha=s(q')$.

In the other direction, let us assume that $qq'\in E(Q^\circ)$ and $b_{p(m')}(s(q'))=1$; hence there exists $x''=[q'',m'']\in V(G^t)$ 
such that $m''\in X_t$, $p(m'')=p(m')$, $s(q'')=s(q')$, and $xx''\in E(G^t)$ by~\eqref{eq:bjdefi}.
Since $qq''\in E(Q^\circ)$ from $G^t\subseteq Q\boxtimes M$, and $s$ is a proper colouring of the square of~$Q$
(in which $q'q''$ is an edge), $s(q'')=s(q')$ implies $q'=q''$.
Since both $m',m''\in X_t$ and $p$ is injective on $X_t$, $p(m'')=p(m')$ likewise implies~$m''=m'$.
So, $x'=x''$ and we have $xx'\in E(G^t)$.
\end{proof}

\Cref{clm:coledge} can be interpreted such that, in order to determine whether \mbox{$xx'\in E(G^t)$}, it is sufficient to know the following information:
whether the parameter vertices $q,q'$ are adjacent, what is the colour $s(q')$ in $Q^2$ which is a part of the initial colour of the vertex $x'$,
and what is the base colour of the vertex $x$ at~$t$ which is maintained throughout the expression in the running colour of~$x$ at the node~$t$.

\begin{claim}\label{clm:colsize}
Let $B_{S,k}$ denote the set of the base colours used in \Cref{sdef:colorx}.
For given $\Delta\geq2$ (bounding $|S|$) and~$k$, we have $|B_{S,k}|\leq (|S|-1)^{(\Delta+1)(k+1)}\leq \Delta^{2(\Delta+1)(k+1)}$.
\end{claim}
\begin{proof}[Subproof.]
Observe that, for each $x=[q,m]\in V_m$ as in \Cref{sdef:colorx} and each $0\leq j\leq k$, 
the function $b_j\in\{0,1\}^{S}$ has at most $\Delta+1$ nonzero values -- one for each vertex in the closed neighbourhood of $q$ in the graph~$Q$ by~\eqref{eq:bjdefi}.
So, let $B_S\subseteq\{0,1\}^{S}$ be the subset of such functions with at most $\Delta+1$ nonzero values.
By simple calculus, we have $|B_S|\leq{|S|\choose0}+{|S|\choose1}+\ldots+{|S|\choose\Delta+1}<(|S|-1)^{\Delta+1}\leq\Delta^{2(\Delta+1)}$ 
for~$\Delta^2+1\geq|S|>\Delta\geq2$.
Altogether, $|B_{S,k}|=|B_S|^{k+1} \leq \Delta^{2(\Delta+1)(k+1)}$.
\end{proof}

Let $\Gamma:=(S\cup\{\perp\})\times B_{S,k}$ be the full set of colours -- the initial and running ones
for colouring vertices in our expressions based on \Cref{sdef:colorx}, and $\ell=|\Gamma|$.

Following the sketch in Items (\ref{it:firstVm}) and (\ref{it:nextPhi}) above,
we now formally give a $(Q^\circ\!,\ell)$-expression $\varphi=\varphi^r$ (where $r$ is the root of~$T$) valued~$G$,
constructed recursively along the nodes $t\in V(T)$ from subexpressions $\varphi^t$ valued $G^t$
where, as defined above, $G^t=G[V(Q)\times Y_t]$.

\begin{claim}\label{clm:recurt}
For every node $t\in V(T)$, there is a $(Q^\circ\!,\ell)$-expression $\varphi^t$ valued $G^t$ such that the following holds.
For each vertex $x=[q,m]\in V(G^t)$, the parameter vertex of $x$ in $\varphi^t$ is~$q$, and if $t\not=r$, 
the resulting colour of $x$ given by $\varphi^t$ equals the running colour (by \Cref{sdef:colorx}) of $x$ at the parent node~$t^\uparrow$ of~$t$.
\end{claim}

\begin{proof}[Subproof.]
We proceed with $t\in V(T)$ by structural induction on the tree $T$, in the leaf-to-root direction, 
and giving the construction of the expression alongside with a proof of correctness
(it may be useful to recall \Cref{fig:treedecoxx} at this point).

\begin{enumerate}[I.]\parskip3pt
\item 
If $t$ is a leaf, then we start with an empty expression $\psi_0$ and $G_0^t=\emptyset$.
If $t$ has one child~$s$, then we take the expression $\psi_0:=\varphi^s$ already constructed at~$s$, and $G_0^t=G^{s}$.
If $t$ has two children~$s,s'$, then we let $\psi_0$ be the union operation over the expressions $\varphi^s$ and~$\varphi^{s'}$,
and $G_0^t=G^{s}\,\dot\cup\,G^{s'}$.
We have the following:
\begin{itemize}
\item The graph $G_0^t$ is an induced subgraph of $G$ by the inductive assumption and, in the last (union) case,
since the graph $M$ has no edges between the sets $Y_s$ and $Y_{s'}$ by the interpolation property of a tree decomposition,
and so there are no edges between the disjoint subgraphs $G^{s}$ and $G^{s'}$ in~$G$.
\item Every vertex $x\in V(G_0^t)$, by the inductive assumption on $\varphi^s$ and possibly~$\varphi^{s'}$,
gets colour in the expression $\psi_0$ equal to the running colour of $x$ at~$t$.
\end{itemize}

\item\label{it:mored}
Then we let $Y'_t:=Y_t\cap X_t$; informally, $Y'_t$ are the vertices of $M$ whose last bag in~$T$ is right at~$t$.
We choose an arbitrary order $Y'_t=(m_1,\ldots,m_a)$, $a=|Y'_t|$; possibly having $a=0$ if $Y'_t=\emptyset$.
For $i=1,\ldots,a$, we repeat the following:
\begin{enumerate}[a)]
\item We start the expression $\psi_i$ by making a union of previous $\psi_{i-1}$ (if nonempty) and of the
expression $\nu_{m_i}$ trivially constructing the edge-less graph on~$V_{m_i}=\big(V(Q)\times\{m_i\}\big)\cap V(G)$.
Each vertex $x=[q,m_i]\in V_{m_i}$ is created in $\nu_{m_i}$ with the parameter vertex $q$ and the initial colour of~$x$ in~$G$ as defined in \Cref{sdef:colorx}.
\item\label{it:added}
Let $j=p(m_i)$. Looping over all $\beta\in S$
and writing, for simplicity, $*$ for an arbitrary value, we add in $\psi_i$ edges between each running colour of the form
$(\perp,\>*,\ldots,b_j,\ldots,*)\in \Gamma$ where $b_j(\beta)=1$ and each initial colour $(\beta,\>*,\ldots,*)\in \Gamma$.

Note that only vertices of $V_{m_i}$ may currently hold in $\psi_i$ colours of the initial type $(\alpha,*,\ldots,*)$ where $\alpha\not=\perp$.
So, by \Cref{clm:coledge} and \Cref{def:hcw}\,\ref{it:Hcwedges}), the described operations in $\psi_i$ create precisely the edges between the sets
$\big(V(G_0^t)\cup V_{m_1}\cup\dots\cup V_{m_{i-1}}\big)$ and $V_{m_i}$ which exist in the graph~$G^t\subseteq_i G$.
\item\label{it:added0}
Similarly to point \ref{it:added});%
\footnote{We may have actually joined points \ref{it:added}) and \ref{it:added0}) into one uniform batch of recolourings, but we prefer to state them separately for clarity.}
looping over all $\beta\in S$, we add in $\psi_i$ edges between each colour of the form $(\alpha,\>*,\ldots,b_j,\ldots,*)\in \Gamma$
where $\alpha\not=\perp$ is arbitrary, $j=p(m_i)$ and $b_j(\beta)=1$, and each colour $(\beta,\>*,\ldots,*)\in \Gamma$.

Again, by \Cref{clm:coledge} and \Cref{def:hcw}\,\ref{it:Hcwedges}), the described operations in $\psi_i$ create precisely the edges
which are within the set $V_{m_i}$ and which exist in~$G^t\subseteq_i G$.
\item\label{it:rperp}
We finish, after \ref{it:added0}), the expression $\psi_i$ with the operations of recolouring every initial colour $c=(\beta,\>b_0,\ldots,b_k)$,
where $\beta\in S$ and $(b_0,\ldots,b_k)\in B_{S,k}$ arbitrary, to the colour $c'=(\perp,\>b_0,\ldots,b_k)$.
This way, as stated in \Cref{sdef:colorx}, every vertex $x\in V_{m_i}$ turns from its initial colour to the running colour of $x$ at~$t$,
which re-establishes the assumptions for the next iteration of point \ref{it:mored}.

Overall, the value of $\psi_i$ is the subgraph of $G$ induced on $\big(V(G_0^t)\cup V_{m_1}\cup\dots\cup V_{m_{i}}\big)$.
\end{enumerate}
\item\label{it:rhash}
Having finished with the $a$ iterations of point \ref{it:mored}, we are present with the expression $\psi_a$ whose value is 
the subgraph of $G$ induced on $\big(V(G_0^t)\cup V_{m_1}\cup\dots\cup V_{m_{a}}\big)$ -- in other words, $\psi_a$ is valued $G^t$ as desired.

To fulfill the remaining condition of \Cref{clm:recurt} -- that each vertex $x\in V(G^t)$ should get the running colour of $x$ at the parent node of~$t$,
it is by \Cref{sdef:colorx} enough to add the following sequence of recolouring operations.
The sought expression $\varphi^t$ results from $\psi_a$ by performing, for every $i=1,\ldots,a$, the operations of recolouring
of every colour $c=(\perp,\>b_0,\ldots,b_j,\ldots,b_k)\in \Gamma$ where $j=p(m_i)$ and $b_0,\ldots,b_k$ are arbitrary,
to colour $c'=(\perp,\>b_0,\ldots,b_j^o,$ $\ldots,b_k)\in \Gamma$ where $b_j^o(\alpha)=0$ for all~$\alpha\in S$.
\qed
\end{enumerate}\let\qedsymbol\relax
\end{proof}%

Finally, at the root $r$ of $T$, the constructed $(Q^\circ\!,\ell)$-expression $\varphi:=\varphi^r$ is valued $G^r=G$ by \Cref{clm:recurt},
as needed in this proof.
The last bit in Part a) is to give an upper bound on $\ell=|\Gamma|=|S\cup\{\perp\}|\cdot|B_{S,k}|\leq (\Delta^2+2)\cdot\Delta^{2(\Delta+1)(k+1)}$ using \Cref{clm:colsize}.

The last bound on $\ell$ can be further improved if we know that the square of $Q$ is $d$-colourable
-- hence we can use $S=\{1,2,\ldots,d\}$, where $d$ can be as low as $\Delta+1$ in the best case.
Then the bound of \Cref{clm:colsize} becomes $\ell=|S\cup\{\perp\}|\cdot|B_{S,k}|\leq (|S|+1)\cdot(|S|-1)^{(\Delta+1)(k+1)}=(d+1)(d-1)^{(\Delta+1)(k+1)}$.
At the same time, there is a trivial upper bound of $|B_{S,k}|\leq2^{|S|(k+1)}$ which gives 
$\ell\leq(|S|+1)\cdot2^{|S|(k+1)}=(d+1)2^{d(k+1)}$.
A combination of these two bounds gives the improved bound stated in the sequel of \Cref{thm:fromprods}.

\medskip

Regarding Part b) of the Theorem -- constructing the graph $M_2$, we remark that we cannot simply employ \Cref{thm:hcwproduct} since 
that would give us only a factor $M_1$ of bounded clique-width as before, but potentially containing unbounded bipartite cliques.
We instead provide an ad-hoc construction of the desired factor $M_2$ which is based on the structure and the initial colours
used in the $(Q^\circ\!,\ell)$-expression $\varphi^r$ valued $G$ from \Cref{clm:recurt}.

Let again $s:V(Q)\to S$ be a proper colouring of the square graph~$Q^2$ of~$Q$.
Let now $\Gamma':=S\times B_{S,k}$. (Comparing this definition to previous $\Gamma$, the only difference is in the omitted symbol~$\perp$.)
Recall also the rooted tree decomposition $(T,\ca X)$ of width $k$ of the graph $M$, and the function $p:V(M)\to\{0,1,\ldots,k\}$ which is injective on each of the bags.
Recall that $Y'_t:=Y_t\cap X_t$ for $t\in V(T)$ denotes the set vertices of the graph $M$ whose last bag in~$T$ (going bottom-up) is right at~$t$.
When defining the expression $\varphi$ of~$G$, we have ordered the members of each $Y'_t$ where $a=|Y'_t|$ as $Y'_t=(m_1,\ldots,m_a)$.

Using this and the native tree order of $T$, we (uniquely) orient all edges of $M$ as follows:
for every $mm'\in E(M)$, which in particular means that $m$ and $m'$ share a bag of the decomposition, we define $m\preceq m'$ if, 
either $m\in Y'_t$ and $m'\in Y'_s$ where $t\not=s$ is a descendant of $s$ in~$T$,
or $m,m'\in Y'_t=(m_1,\ldots,m_a)$ and $m=m_i$, $m'=m_j$ for some~$i\leq j$.

We are going to define a graph $M_2$ such that $V(M_2):=V(M)\times\Gamma'$, 
and $E(M_2)\subseteq F$ where $F=\{\{(m,\gamma),(m',\gamma')\}: m=m'\,\vee\,mm'\in E(M),\>\gamma,\gamma'\in\Gamma'\}$.
This setting, in particular, by \Cref{clm:colsize} clearly implies that the tree-width of $M_2$ is going to be at most 
\begin{equation}\label{eq:colsettw}
\mathop{tw}(M_2)\leq(k+1)\cdot|\Gamma'|-1\leq (k+1)|S|\cdot|B_{S,k}| \leq (k+1)(\Delta^2+1)\cdot\Delta^{2(\Delta+1)(k+1)}
.\end{equation}
The edge set $E(M_2)\subseteq F$ is defined precisely as follows. Let $m,m'\in V(M)$ be such that $m=m'$ or $mm'\in E(M)$, and $m\preceq m'$. 
Let $\gamma,\gamma'\in\Gamma'$ be such that $\gamma=(\alpha,b_0,\ldots,b_k)$ and $\gamma'=(\alpha',b'_0,\ldots,b'_k)$, and $(m,\gamma)\not=(m',\gamma')$. Then%
\begin{equation}\label{eq:edgeM2}
 \mbox{$\{(m,\gamma),(m',\gamma')\}\in E(M_2)$, if and only if $b_{j}(\alpha')=1$ where $j=p(m')$}.
\end{equation}
(The readers are encouraged to compare this definition to \eqref{eq:bjdefi} above.)

\medskip
The last task is to identify an isomorphism of $G\subseteq Q\boxtimes M$ to an induced subgraph of the product $Q\boxtimes M_2$.
To each vertex $x\in V(G)$ such that $x=[q,m]$ in $Q\boxtimes M$, we assign a vertex of $Q\boxtimes M_2$
by $\iota(x)=[q,(m,\gamma_x)]$ where $\gamma_x=(s(q),c_x)\in\Gamma'$ is given by \Cref{sdef:colorx} as the initial colour of $x$ in~$G$.
Let $G'$ denote the induced subgraph of $Q\boxtimes M_2$ which is the image of~$\iota$.
We now prove that $\iota$ is an isomorphism of $G$ to~$G'$.

Observe that all potential edges of $G$ and of $G'$ are of the form $e=xx'=\{[q,m],[q',m']\}$ in $G$ 
and $\iota$-corresponding $e'=\{[q,(m,\gamma_{x})],[q',(m',\gamma_{x'})]\}$ in the image~$G'$, such that, by the definition of~$\boxtimes$,
$qq'\in E(Q^\circ)$, and $mm'\in E(M)$ or~$m=m'$.
Without loss of generality, we will thus further assume $qq'\in E(Q^\circ)$, and $m\preceq m'$ up to symmetry.

By \Cref{clm:coledge}, $xx'\in E(G^t)=E(G)$ if and only if $b_{p(m')}(s(q'))=1$ where $b_{p(m')}\in \gamma_{x}$.
Since $G'\subseteq_i Q\boxtimes M_2$, we have that $e'\in E(G')$ if and only if $\{(m,\gamma_{x}),(m',\gamma_{x'})\}\in E(M_2)$.
By~\eqref{eq:edgeM2}, $\{(m,\gamma_{x}),(m',\gamma_{x'})\}\in E(M_2)$ if and only if $b_{p(m')}(\alpha')=1$, 
where $\gamma_{x'}=(\alpha',*)=(s(q'),*)$ by the definition of~$\iota$.
Altogether, $e=xx'\in E(G)$ if and only if $e'\in E(G')$, as required.

Together with the bound \eqref{eq:colsettw} on the tree-width of $M_2$, the proof of Part b) is finished.
Although, again as in Part a), the bound \eqref{eq:colsettw} can be further improved if we know that the square of the graph $Q$ is $d$-colourable,
using the set $S=\{1,2,\ldots,d\}$.
Then we get $\mathop{tw}(M_2)\leq (k+1)|S|\cdot|B_{S,k}|\leq (k+1)d(d-1)^{(\Delta+1)(k+1)}$ and, at the same time,
$\mathop{tw}(M_2)\leq (k+1)d2^{d(k+1)}$.
These, combined together, give the improved bound stated in the sequel of \Cref{thm:fromprods}.
\end{proof}

\begin{corollary}\label{cor:fromprods}
Assume that a graph $G$ is a subgraph (not necessarily induced) of the strong product $G\subseteq P\boxtimes M$ where $P$ is a path 
and $M$ is a simple graph of tree-width at most~$k$. 
Then there exists a graph $M_1$ of clique-width at most $\ell=4\cdot8^{k+1}=2^{3k+5}$,
and a graph $M_2$ of tree-width at most $3(k+1)\cdot8^{k+1}$, such that $G$ is isomorphic to {\em induced} subgraphs
of each of the strong products~$P\boxtimes M_1$ and~$P\boxtimes M_2$.
In particular, $\hcw{\ca P^\circ\!}(G)\leq\ell=2^{3k+5}$.
\end{corollary}
\begin{proof}
We observe that the square of $P$ is always $3$-colourable (``modulo~3''), and so we use the improved bounds of \Cref{thm:fromprods} with $d=3$.
\end{proof}

\section{Induced Planar Graph Product Structure}
\label{sec:induplanar}

One may see that the price we pay for the generality of \Cref{thm:fromprods} is substantial;
in order to obtain $G$ as an induced subgraph of the strong product $Q\boxtimes M_2$, the tree-width
of $M_2$ can be up to exponential in the original tree-width of $M$.
In the case of the traditional planar product structure one would get from \Cref{cor:fromprods} directly
a factor of tree-width up to possibly $2^{23}=8\,388\,608$ which is huge.
However, the case of planar graphs is quite special, and with an ad-hoc proof we get the tree-width bound down a lot.

\begin{theorem}\label{thm:onlyplanar}
Every simple planar graph is isomorphic to an \emph{induced} subgraph of the strong product $P\boxtimes M$ where $P$ is a path and $M$ is of tree-width at most~$39$.
\end{theorem}

\begin{proof}
Let $\bar G$ be any simple planar graph and $G$ be a planar triangulation such that $\bar G\subseteq_i G$,
where $G$ is obtained by suitably adding vertices and edges into faces of $\bar G$.
By transitivity of the induced subgraph relation, it suffices to prove the statement for the triangulation~$G$.

If $G$ is a graph and $T\subseteq G$ is an arbitrary BFS (breadth-first-search) tree of $G$, then a path $Q\subseteq G$
is called {\em$T$-vertical} if $Q$ is a subpath of some leaf-to-root path of~$T$.
If $T$ is implicitly clear from the context, we shortly say that $Q$ is {\em vertical} in~$G$.
For a start, we fix an embedding of $G$ in the plane, and associate $G$ with one of its BFS trees $T$ rooted in a vertex of the outer triangular face.

Our proof suitably adapts some core ideas of proofs of the traditional planar graph product structure \cite{DBLP:journals/jacm/DujmovicJMMUW20,DBLP:journals/combinatorics/UeckerdtWY22}.
The cornerstone is the following technical claim which carefully refines (mainly in \Cref{clm:6paths}\,\ref{it:nbatends})) analogous claims of previous proofs of the Planar graph product structure theorem.
See also \Cref{fig:recdecomp6}.
If $C\subseteq G$ is a cycle, we denote by $U_C\subseteq V(G)$ the subset of those vertices of $G$ which are embedded inside the bounded face of~$C$
(in particular, $U_C\cap V(C)=\emptyset$).
Note that $G[U_C]$ then denotes the subgraph of $G$ induced by the vertices embedded inside~$C$.

\begin{claim}\label{clm:6paths}
Let $C\subseteq G$ be a cycle and assume that there are $6$ disjoint paths $Q_i\subseteq C$ for $1\leq i\leq 6$ 
-- some of them possibly being one-vertex or empty, 
such that each nonempty $Q_i$ is vertical in~$G$ and they together partition the vertex set of $C$; $V(Q_1)\cup\ldots\cup V(Q_6)=V(C)$.
If $U_C\not=\emptyset$, then there exist another three disjoint paths $Q_j\subseteq G[U_C]$ for $7\leq j\leq 9$ which are vertical in~$G$,
and again possibly one-vertex or empty, such that the following properties~hold:
\begin{enumerate}[a)]
\item\label{it:dividD}
Denoting by $W:=V(C)\cup V(Q_7\cup Q_8\cup Q_9)$, for every vertex $u\in U_C\setminus W$
there exists a cycle $D_u\subseteq G[W]$ in the subgraph induced by $W$ such that $D_u\not=C$ and $D_u$ satisfies the following;
\begin{itemize}
\item the edges $E(D_u)\setminus E(C)$ constitute a nonempty path embedded inside the cycle~$C$,
\item $u\in U_{D_u}$ and $U_{D_u}\cap W=\emptyset$ (that is, the vertex $u$ is embedded inside the cycle $D_u$, but no vertex of $W$ is embedded inside $D_u$), and
\item $D_u$ intersects at most $6$ of the paths $Q_i$, $1\leq i\leq 9$, each in one connected subpath.
\end{itemize}
\item\label{it:nbatends}
For every $j\in\{7,8,9\}$, if a vertex $v\in V(Q_j)$ is adjacent to any vertex in $V(Q_1\cup\ldots\cup Q_{j-1})$, 
then $v$ is one of the (at most two) ends of~$Q_j$.
\item\label{it:noQ9}
If $Q_9\not=\emptyset$, then there exists $1\leq h\leq 6$ such that $G$ has no edge between $V(Q_9)$ and $V(Q_h)$ and,
for every cycle $D_u$ from \ref{it:dividD}), $Q_9\cap D_u=\emptyset$ or $Q_h\cap D_u=\emptyset$.
\end{enumerate}
The paths $Q_7,Q_8,Q_9$ obtained above will be called {\em primeval vertical paths} of~$C$ in~$G$, 
and each of the cycles $D_u$ from \ref{it:dividD}) will be called a {\em child cycle} of~$C$ in~$G$.
\end{claim}

\begin{proof}[Subproof.]
We first consider a degenerate case -- that there exists an edge $f\in E(G)$ with both ends in $V(C)$ and $f$ embedded inside the bounded face of~$C$.
Let $D_0,D_1$ be the two cycles sharing $f$ in $C+f$, the graph obtained by adding $f$ to the cycle~$C$.
Then we may choose empty paths $Q_7=Q_8=Q_9=\emptyset$, so $W=V(C)$, and \ref{it:dividD}) will be trivially satisfied for every component of $G[U_C]$
with either $D_0$ or~$D_1$, while the rest is void now.
Hence we further assume that no such chord $f$ with both ends in $C$ exists and, consequently, that the induced subgraph $G[U_C]$
(the one drawn inside~$C$) is connected since $G$ is a triangulation.

Up to symmetry, we may assume that for some $m\leq6$ the paths $Q_1,\ldots,Q_m$ are nonempty and occurring in this cyclic order on~$C$, while
$Q_{i}=\emptyset$ for~$m<i\leq 6$.
We may also, without loss of generality, assume that $m\geq3$, or we artificially and arbitrarily split the path $Q_1$ (or~$Q_2$) into up to three subpaths.

Recall that the root of our BFS tree $T$ is outside of $U_C$.
For every vertex $u\in U_C$, there is a unique vertical path $R\subseteq G[U_C]\cap T$ starting in $u$ and ending in a vertex $w\in U_C$, 
such that $w$ is the {\em only} vertex of $R$ adjacent to~$V(C)$.
We denote it by $R[u]=R$, define $\iota(u)\in\{1,\ldots,m\}$ as the least index for which the end $w$ of $R[u]$ is 
adjacent to some $t\in V(Q_{\iota(u)})$ where $t$ is chosen arbitrarily from $V(Q_{\iota(u)})$, 
and set $R^+[u]=R[u]+wt$ which gives that $R^+[u]$ ends on~$C$.
For each $i\in\{1,\ldots,m\}$ and every $u\in V(Q_i)$ we set $\iota(u)=i$, too, and we will further call $\iota(u)$ the {\em colour} of~$u$.
We also, for technical reasons, define $R[u]=\emptyset$ and $R^+[u]=\{u\}$ for all~$u\in V(C)$.
Note that not every colour among $1,\ldots,m$ necessarily has to occur in~$U_C$ with this definition.
(At this point we slightly differ from the choice of vertex colours of $U_C$ in previous similar proofs, 
such as in \cite{DBLP:journals/jacm/DujmovicJMMUW20,DBLP:journals/combinatorics/UeckerdtWY22}, and we do so with the intention to fulfill condition \ref{it:nbatends}).)

Observe that the subset~$X_i\subseteq V(C)\cup U_C$ induced by each colour $i\in\{1,\ldots,m\}$ as above is connected in $G$ by the definition of~$\iota$.
We can hence obtain a minor $G'$ of $G$ by contracting each $X_i$ into one vertex $x$ of colour $\iota(x)=i$.
Since we assume simple graphs, loops and parallel edges created by the contractions are deleted in $G'$,
and since $G$ is a triangulation, so is~$G'$.
In particular, the cycle $C$ is contracted down to an $m$-cycle $C'\subseteq G'$,
and all edges incident to $U_C$ in $G$ are removed in $G'$ or contracted down to $m-3$ chords triangulating the bounded face of~$C'$ if~$m\geq4$.
Let $B'$ denote $C'$ together with the latter $m-3$ chords.

If $m\in\{4,5\}$, we choose $x_7x_8\in E(B')$ such that it is one of the chords, that is, $x_7x_8\in E(B')\setminus E(C')$.
If $m=3$, we choose $x_7x_8\in E(B')$ arbitrarily.
Since $B'$ is obtained by contractions, there exist $u_7,u_8\in V(C)\cup U_C$ such that $\iota(u_7)=\iota(x_7)$, $\iota(u_8)=\iota(x_8)$ and $u_7u_8\in E(G)$.
Among all such edges $u_7u_8\in E(G)$ satisfying $\iota(u_7)=\iota(x_7)$, $\iota(u_8)=\iota(x_8)$ we choose one minimising the size of $R^+[u_7]\cup R^+[u_8]$.
Note that $R[u_7]\cup R[u_8]\not=\emptyset$, since otherwise~$\{u_7,u_8\}\subseteq V(C)$ which is the degenerate case handled at the beginning.
By our minimal choice, 
the vertices of $V(R[u_8])\setminus\{u_8\}$ are not adjacent to $R[u_7]$,
and hence we fulfill condition~\ref{it:nbatends}) with $Q_7=R[u_7]$, $Q_8=R[u_8]$ and $Q_9=\emptyset$.
Condition \ref{it:noQ9}) is void.
Regarding condition~\ref{it:dividD}), for each $u$ we can clearly choose one of the two cycles $D_0,D_1\subseteq G[W]$
which ``halve'' $C$ through the path $R^+[u_7]\cup R^+[u_8]$,
and each of~$D_0$ and $D_1$ intersects at most $4$ of the paths $Q_1,\ldots,Q_m$ plus the two paths~$Q_7,Q_8$.

\begin{figure}[t]
$$
\begin{tikzpicture}[xscale=0.5, yscale=0.6]
  \tikzstyle{every node}=[draw, black, shape=circle, minimum size=2.8pt,inner sep=0pt, fill=black]
  \tikzstyle{every path}=[draw, color=black]
  \draw[line width=1.2pt,dotted] (-0.5,0)--(0.5,0) to [bend right=20] (3,1.5)--(3.5,2.2) to [bend right=20] (3.5,5)--(3,5.7) to [bend right=20] (0.5,7)
	--(-0.5,7) to [bend right=20] (-3,5.7)--(-3.5,5) to [bend right=20] (-3.5,2.2)--(-3,1.5) to [bend right=20] (-0.5,0) ;
  \draw[->] (0.5,0) node{} to [bend right=20] (3,1.5) node{} to +(0.5,0);
  \draw[->] (3.5,2.2) node{} to [bend right=20] (3.5,5) node{} to +(0.3,0.3);
  \draw[->] (3,5.7) node{} to [bend right=20] (0.5,7) node{} to +(0,0.5);
  \draw[->] (-0.5,0) node{} to [bend left=20] (-3,1.5) node{} to +(-0.5,0);
  \draw[->] (-3.5,2.2) node{} to [bend left=20] (-3.5,5) node{} to +(-0.3,0.3);
  \draw[->] (-3,5.7) node{} to [bend left=20] (-0.5,7) node{} to +(0,0.5);
  \node[draw=none,fill=none] at (-4.4,6.5) {$C$};
  \node[draw=none,fill=none] at (-2,7) {$Q_3$};
  \node[draw=none,fill=none] at (-4.4,4) {$Q_2$};
  \node[draw=none,fill=none] at (-2,0.15) {$Q_1$};
  \node[draw=none,fill=none] at (2,7) {$Q_4$};
  \node[draw=none,fill=none] at (4.4,4) {$Q_5$};
  \node[draw=none,fill=none] at (2.1,0.15) {$Q_6$};
  \node[draw=none,fill=none] at (-0.4,1.6) {$Q_7$};
  \node[draw=none,fill=none] at (2,2.8) {$Q_8$};
  \node[draw=none,fill=none] at (-2,4.25) {$Q_9$};
  \tikzstyle{every node}=[draw, black, shape=circle, minimum size=2.2pt,inner sep=0pt, fill=white]
  \footnotesize
  \draw[draw=none,fill=gray!20] (0,4.6)--(0.6,5.2)--(-0.6,5.2)--cycle;
  \node[label=right:$\!u_7'$] at (0,4.6) (uu7) {};  \node[label=above:$u_8'$] at (0.6,5.2) (uu8) {};  \node[label=above:$u_9'$] at (-0.6,5.2) (uu9) {};
  \tikzstyle{every node}=[draw, black, shape=circle, minimum size=2.8pt,inner sep=0pt, fill=black]
  \node[label=right:$u_7$] at (0,3.8) (u7) {};  \node[label=below:$u_8$] at (0.8,2.5) (u8) {};
  \node[label=left:$v_8\>$] at (-0.2,2.5) (v8) {};  \node[label=below:$u_9$] at (-0.8,3.9) (u9) {};
  \draw[dashed,->,>=stealth] (uu7)--(uu8)--(uu9)--(uu7)--(u7);
  \draw[dashed,->,>=stealth] (uu8) to[bend left=14] (u8);  \draw[dashed,->,>=stealth] (uu9) to[bend left=14] (u9);
  \tikzstyle{every path}=[draw,thick, color=black]
  \tikzstyle{every node}=[draw, black, shape=circle, minimum size=2.5pt,inner sep=0pt, fill=black]
  \draw (u7) to[bend left=12] (v8) to[bend left=12] (-1.9,1.1) node{} --(-2.5,0.95) node{};
  \draw (u8) to[bend right=16] (3.1,2.7) node{} --(3.7,2.9) node{};
  \draw (u9) to[bend left] (-2.2,5.7) node{} --(-2.5,6.2) node{};
  \draw[dashed] (u7)--(u9) (u8)--(v8) ;
  \tikzstyle{every path}=[draw,thick, color=black]
  \tikzstyle{every node}=[draw, white, shape=circle, minimum size=2.2pt,inner sep=0pt, fill=white]
  \node[white] at (-2.2,1) {};  \node[white] at (3.4,2.8) {};  \node[white] at (-2.36,5.95) {};
\end{tikzpicture} 
\qquad
\begin{tikzpicture}[xscale=0.5, yscale=0.6]
  \tikzstyle{every node}=[draw, black, shape=circle, minimum size=2.8pt,inner sep=0pt, fill=black]
  \tikzstyle{every path}=[draw, color=black]
  \draw[line width=1.2pt,dotted] (-0.5,0)--(0.5,0) to [bend right=20] (3,1.5)--(3.5,2.2) to [bend right=20] (3.5,5)--(3,5.7) to [bend right=20] (0.5,7)
	--(-0.5,7) to [bend right=20] (-3,5.7)--(-3.5,5) to [bend right=20] (-3.5,2.2)--(-3,1.5) to [bend right=20] (-0.5,0) ;
  \draw[->] (0.5,0) node{} to [bend right=20] (3,1.5) node{} to +(0.5,0);
  \draw[->] (3.5,2.2) node{} to [bend right=20] (3.5,5) node{} to +(0.3,0.3);
  \draw[->] (3,5.7) node{} to [bend right=20] (0.5,7) node{} to +(0,0.5);
  \draw[->] (-0.5,0) node{} to [bend left=20] (-3,1.5) node{} to +(-0.5,0);
  \draw[->] (-3.5,2.2) node{} to [bend left=20] (-3.5,5) node{} to +(-0.3,0.3);
  \draw[->] (-3,5.7) node{} to [bend left=20] (-0.5,7) node{} to +(0,0.5);
  \node[draw=none,fill=none] at (-4.4,6.5) {$C$};
  \node[draw=none,fill=none] at (-2,7) {$Q_3$};
  \node[draw=none,fill=none] at (-4.4,4) {$Q_2$};
  \node[draw=none,fill=none] at (-2,0.15) {$Q_1$};
  \node[draw=none,fill=none] at (2,7) {$Q_4$};
  \node[draw=none,fill=none] at (4.4,4) {$Q_5$};
  \node[draw=none,fill=none] at (2.1,0.15) {$Q_6$};
  \node[draw=none,fill=none] at (-0.4,1.6) {$Q_7$};
  \node[draw=none,fill=none] at (2.4,3.3) {$Q_8$};
  \node[draw=none,fill=none] at (-2.4,3.5) {$Q_9$};
  \tikzstyle{every node}=[draw, black, shape=circle, minimum size=2.2pt,inner sep=0pt, fill=white]
  \footnotesize
  \draw[draw=none,fill=gray!20] (0,4.6)--(0.6,5.2)--(-0.6,5.2)--cycle;
  \node[label=right:$\!u_7'$] at (0,4.6) (uu7) {};  \node[label=above:$u_8'$] at (0.6,5.2) (uu8) {};  \node[label=above:$u_9'$] at (-0.6,5.2) (uu9) {};
  \tikzstyle{every node}=[draw, black, shape=circle, minimum size=2.8pt,inner sep=0pt, fill=black]
  \node[label=left:$u_7$] at (0,3.8) (u7) {};  \node[label=right:$\,u_8$] at (0.8,3.8) (u8) {};
  \node[label=right:$\,v_9$] at (-0.2,2.5) (v8) {};  \node[label=left:$u_9\>$] at (-1,2.7) (u9) {};
  \draw[dashed,->,>=stealth] (uu7)--(uu8)--(uu9)--(uu7)--(u7);
  \draw[dashed,->,>=stealth] (uu8) to[bend left=14] (u8);  \draw[dashed,->,>=stealth] (uu9) to[bend right=4] (u9);
  \tikzstyle{every path}=[draw,thick, color=black]
  \tikzstyle{every node}=[draw, black, shape=circle, minimum size=2.5pt,inner sep=0pt, fill=black]
  \draw (u7) to[bend left=12] (v8) to[bend left=12] (-1.9,1.1) node{} --(-2.5,0.95) node{};
  \draw (u8) to[bend right=16] (3.1,2.7) node{} --(3.7,2.9) node{};
  \draw (u9) to[bend left] (-2.2,5.7) node{} --(-2.5,6.2) node{};
  \draw[dashed] (v8)--(u9) (u8)--(u7) ;
  \tikzstyle{every path}=[draw,thick, color=black]
  \tikzstyle{every node}=[draw, white, shape=circle, minimum size=2.2pt,inner sep=0pt, fill=white]
  \node[white] at (-2.2,1) {};  \node[white] at (3.4,2.8) {};  \node[white] at (-2.36,5.95) {};
\end{tikzpicture} 
$$ $$
\begin{tikzpicture}[xscale=0.5, yscale=0.6]
  \tikzstyle{every node}=[draw, black, shape=circle, minimum size=2.8pt,inner sep=0pt, fill=black]
  \tikzstyle{every path}=[draw, color=black]
  \draw[line width=1.2pt,dotted] (-0.5,0)--(0.5,0) to [bend right=20] (3,1.5)--(3.5,2.2) to [bend right=20] (3.5,5)--(3,5.7) to [bend right=20] (0.5,7)
	--(-0.5,7) to [bend right=20] (-3,5.7)--(-3.5,5) to [bend right=20] (-3.5,2.2)--(-3,1.5) to [bend right=20] (-0.5,0) ;
  \draw[->] (0.5,0) node{} to [bend right=20] (3,1.5) node{} to +(0.5,0);
  \draw[->] (3.5,2.2) node{} to [bend right=20] (3.5,5) node{} to +(0.3,0.3);
  \draw[->] (3,5.7) node{} to [bend right=20] (0.5,7) node{} to +(0,0.5);
  \draw[->] (-0.5,0) node{} to [bend left=20] (-3,1.5) node{} to +(-0.5,0);
  \draw[->] (-3.5,2.2) node{} to [bend left=20] (-3.5,5) node{} to +(-0.3,0.3);
  \draw[->] (-3,5.7) node{} to [bend left=20] (-0.5,7) node{} to +(0,0.5);
  \node[draw=none,fill=none] at (-4.4,6.5) {$C$};
  \node[draw=none,fill=none] at (-2,7) {$Q_3$};
  \node[draw=none,fill=none] at (-4.4,4) {$Q_2$};
  \node[draw=none,fill=none] at (-2,0.15) {$Q_1$};
  \node[draw=none,fill=none] at (2,7) {$Q_4$};
  \node[draw=none,fill=none] at (4.4,4) {$Q_5$};
  \node[draw=none,fill=none] at (2.1,0.15) {$Q_6$};
  \node[draw=none,fill=none] at (-0.9,1.6) {$Q_7$};
  \node[draw=none,fill=none] at (1.6,2.8) {$Q_8$};
  \node[draw=none,fill=none] at (0,5.4) {$Q_9$};
  \tikzstyle{every node}=[draw, black, shape=circle, minimum size=2.2pt,inner sep=0pt, fill=white]
  \footnotesize
  \draw[draw=none,fill=gray!20] (-0.6,4.3)--(-0.1,3.4)--(0.8,3.9)--cycle;
  \node[label=left:$u_7'$] at (-0.6,4.3) (uu7) {};  \node[label=left:$u_8'\!$] at (-0.1,3.4) (uu8) {};  \node[label=below:$u_9'$] at (0.8,3.9) (uu9) {};
  \tikzstyle{every node}=[draw, black, shape=circle, minimum size=2.8pt,inner sep=0pt, fill=black]
  \node[label=below:$u_8$] at (-0.1,2.8) (u8) {};  \node[label=left:$u_7\>$] at (-1.2,2.8) (u7) {};
  \node[label=right:$\,u_9$] at (1.6,4.1) (u9) {};
  \draw[dashed,->,>=stealth] (uu7)--(uu8)--(uu9)--(uu7)--(u7);
  \draw[dashed,->,>=stealth] (uu8) -- (u8);  \draw[dashed,->,>=stealth] (uu9) -- (u9);
  \tikzstyle{every path}=[draw,thick, color=black]
  \tikzstyle{every node}=[draw, black, shape=circle, minimum size=2.5pt,inner sep=0pt, fill=black]
  \draw (u7) to[bend left=12] (-1.9,1.3) node{} --(-2.5,0.95) node{};
  \draw (u8) to[bend right=12] (3.1,4.1) node{} --(3.64,4.5) node{};
  \draw (u9) to[bend right] (-1.5,6.3) node{} --(-2.2,6.4) node{};
  \draw[dashed] (u7)--(u8) (u9)--(1.8,3.37) node[label=right:$\>v_9$] {};
  \tikzstyle{every path}=[draw,thick, color=black]
  \tikzstyle{every node}=[draw, white, shape=circle, minimum size=2.2pt,inner sep=0pt, fill=white]
  \node[white] at (-2.2,1.1) {};  \node[white] at (3.4,4.3) {};  \node[white] at (-1.9,6.35) {};
\end{tikzpicture} 
$$
\caption{Illustrating \Cref{clm:6paths}; three of the possibilities of a recursive decomposition of 
	the interior of the cycle $C$ along vertical paths $Q_7,Q_8,Q_9$ for $m=6$ are shown in the pictures. 
	The cycle $C$ is drawn using heavy dots, and the shaded triple $(u_7',u_8',u_9')$
	is the ``tricoloured'' triangular face of $G$ that we start with in the proof.
	$Q_7$, $Q_8$ and $Q_9$ are the selected subpaths of the vertical paths starting in $u_7'$, $u_8'$ and~$u_9'$ in this order.
	Lengths and marked orientations towards the root of the vertical paths $Q_1,\ldots,Q_6$ making up $C$ are not important for this illustration.}
\label{fig:recdecomp6}
\end{figure}

\smallskip
Now let~$m=6$.
By a short case analysis of the contracted graph $B'$ defined above, 
there exists a triangle $B'_0\subseteq B'$ sharing at most one edge with $C'$, more precisely,
$V(B'_0)=\{x_7,x_8,x_9\}$ and $E(B'_0)\cap E(C')\subseteq\{x_8x_9\}$. 
Then, up to symmetry between our vertices, the distance between $x_8$ and $x_9$ on $C'$ is one or two, 
the distance between $x_7$ and $x_9$ on $C'$ equals two and the distance between $x_7$ and $x_8$ on $C'$ is two or three.
Similarly to the previous paragraph, by simultaneous ``uncontraction'', there exist a vertex triple $u'_7,u'_8,u'_9\in V(C)\cup U_C$ such that 
$\iota(u'_i)=\iota(x_i)$ for $i\in\{7,8,9\}$ and $(u'_7,u'_8,u'_9)$ bounds a triangular face of~$G$.

We choose vertices $u_8$ on $R^+[u_8']$ and $u_9$ on $R^+[u_9']$ lexicographically minimising
the lengths of the paths $R^+[u_8]$ in the first place, and of $R^+[u_9]$ in the second place,
subject to the following; $\iota(u_8)=\iota(x_8)$, $\iota(u_9)=\iota(x_9)$,
$u_8$ is adjacent to a vertex $v_8\in V(R^+[u'_7])$ and $u_9$ is adjacent to a vertex $v_9\in V(R[u'_7]\cup R[u_8])$.
In particular, we have $v_9\not\in V(C)$ which will be used further.

Note that the previous definition of $u_9$ as a neighbour of some $v_9\in V(R[u'_7]\cup R[u_8])$ is sound 
since $R[u'_7]\cup R[u_8]\not=\emptyset$ holds similarly as in the case of~$m\leq5$.
If there was a vertex $x\in V(R[u_8])\setminus\{u_8\}$ adjacent to $R[u'_7]$, we would contradict our minimal choice with $u_8=x$ and $u_9=u'_9$.
Hence there is no such $x\not=u_8$, and we analogously argue that there is no vertex of $V(R[u_9])\setminus\{u_9\}$ adjacent to~$R[u'_7]\cup R[u_8]$.
Finally, we choose $u_7=v_8$, or $u_7=v_9$ if $v_9\in V(R[u'_7])$ is closer to $u'_7$ than~$v_8$~is.
So, the choice of $Q_7=R[u_7]$, $Q_8=R[u_8]$ and $Q_9=R[u_9]$ fulfills condition~\ref{it:nbatends}) of the lemma.
See~\Cref{fig:recdecomp6}.

To address condition~\ref{it:dividD}), we consider the graph $F:=(C\cup R^+[u_7]\cup R^+[u_8]\cup R^+[u_9])+u_8v_8+u_9v_9\subseteq G$ 
and observe that $F$ is formed by $C$ and by exactly three internally disjoint paths starting in a vertex $s\in\{v_8,v_9\}$
and denoted by $S_7$, $S_8$ and $S_9$ such that they arrive at~$Q_{\iota(u_7)}$, $Q_{\iota(u_8)}$ and $Q_{\iota(u_9)}$, respectively.
More specifically, if $u_7=v_8$, we have $s=v_9$ and $S_9=R^+[u_9]+u_9v_9$ and $S_7,S_8\subseteq(R^+[u_7]\cup R^+[u_8])+u_8v_8$,
and if $u_7=v_9$, we have $s=v_8$ and $S_8=R^+[u_8]+u_8v_8$ and $S_7\subseteq R^+[u_7]$, $S_9\subseteq(R^+[u_7]\cup R^+[u_9])+u_9v_9$.

The cycle $D_0\subseteq F$ containing $S_7\cup S_8$, by our choice of the vertices $x_7,x_8$, intersects $3$ or $4$ of the paths $Q_1,\ldots,Q_6$
and the paths $Q_7,Q_8$ (but not $Q_9$), which is at most $4+2=6$ as required by condition~\ref{it:dividD}).
Furthermore, since $D_0$ hits at least three of the paths on $C$, 
there is $1\leq h\leq6$ such that $D_0$ intersects $Q_{h}$ and $h\not\in\{\iota(u_7),\iota(u_8)\}$.
Then each of the two remaining cycles $D_1,D_2\subseteq F$ containing $S_9$ intersects $2$ or $3$ of the paths $Q_1,\ldots,Q_6$ by our choice of $x_7,x_8,x_9$,
and possibly all three of $Q_7,Q_8,Q_9$, which is again at most $3+3=6$.
And since $Q_9=R[u_9]\subseteq S_9$ is (strictly) separated from $Q_{h}$ by the cycle $D_0\Delta C$ in the embedding of~$G$,
there cannot exist an edge between $V(Q_9)$ and $V(Q_{h})$.
Condition~\ref{it:noQ9}) is fulfilled as well.
\end{proof}

With \Cref{clm:6paths} at hand, the rest of the proof is a straightforward induction starting from the outer face of $G$ and continuing ``inward''
by recursive applications of \Cref{clm:6paths}.
To construct the factor $M$, in rough sketch, we assign vertices of $M$ to each of the primeval vertical paths in $G$, 
and we build a tree decomposition of $M$ using bags formed according to the individual applications of \Cref{clm:6paths}.
Unlike in proofs of the traditional planar graph product structure, where the correspondence of the primeval vertical paths to the vertices of $M$
is one-to-one, we now assign specially crafted $5$-tuples of vertices of $M$ to each of the primeval vertical paths in $G$, 
and this helps us to express $G$ as an induced subgraph of the product.

To formulate the above sketch precisely, we need appropriate notation.
For a graph $G$ with a BFS tree $T\subseteq G$ rooted at $r\in V(G)$, let the {\em level} of a vertex $x\in V(G)$ -- denoted by $\ell_T(x)$ -- 
be the distance from $x$ to $r$ in~$T$ (which equals this distance in~$G$).
Note that the levels of adjacent vertices differ by at most $1$, and for
the coming construction it is crucial that the levels of the vertices of any $T$-vertical path strictly monotonously decrease in the direction towards the root.
Furthermore, let $\ca Q$ be a family of pairwise disjoint $T$-vertical paths in~$G$ covering all vertices, i.e., $\bigcup_{Q\in\ca Q}V(Q)=V(G)$.
For $\ca Q_0\subseteq\ca Q$, we say that a subgraph $G_0\subseteq_i G$ is {\em induced by $\ca Q_0$} 
if $G_0$ is the subgraph induced on the vertex subset~$\bigcup_{Q\in\ca Q_0}V(Q)$.

We say that a product $P\boxtimes M$, where $P$ is a path on the vertices $(p_0,p_1,\ldots,p_m)$ in this order for suitably large~$m$, 
is a {\em$t$-fold nice product structure} of $(G,\ca Q)$ if $G$ is isomorphic to an induced subgraph of $P\boxtimes M$ such that the following holds:
\begin{enumerate}[(i)]
\item $V(M)=\ca Q\times\{1,\ldots,t\}$,
\item\label{it:pmvert}
every vertex $v\in V(G)$ where $v\in V(Q)$ for $Q\in\ca Q$, is represented in an induced subgraph of $P\boxtimes M$ by a vertex
	$[p_i,m_v]\in V(P\boxtimes M)$ such that $i=\ell_T(v)$ and $m_v=(Q,j)$ for some $j\in\{1,\ldots,t\}$, and
\item\label{it:threedisti} 
for every three consecutive vertices $v,u,w$ on a path $Q\in\ca Q$, we have in \eqref{it:pmvert} $m_v=(Q,j)$, $m_u=(Q,j')$ and $m_w=(Q,j'')$ such that $j$, $j'$ and $j''$ are pairwise distinct.
\end{enumerate}
Furthermore, for a factor $M$ as above,
we say that a tree decomposition $(T_1,\ca X)$ of the graph $M$ is {\em $\ca Q$-aligned of thickness $k$}
if every bag in $\ca X$ is of the form $\ca Q_0\times\{1,\ldots,t\}$ where $\ca Q_0\subseteq\ca Q$ and $|\ca Q_0|\leq k$.
We also say that $(T_1,\ca X)$ {\em exposes} a subfamily $\ca Q_1\subseteq\ca Q$ if there is a bag $X\in\ca X$ such that $\ca Q_1\times\{1,\ldots,t\}\subseteq X$.

Now we formulate:
\begin{claim}\label{clm:6recur}
Let $G_1\subseteq_i G$ be a subgraph of $G$ induced by a family of vertical paths $\ca Q_1$ in~$G$.
Let $C\subseteq G$ be a cycle such that $C$ is a facial cycle in the subembedding of $G_1$, 
formally written as $G_1\supseteq C$ and $V(G_1)\cap U_C=\emptyset$,
and that all paths of~$\ca Q_1$ intersecting~$C$ form~a~subset $\ca Q_0\subseteq\ca Q_1$ of size $|\ca Q_0|\leq6$.
Assume that $P\boxtimes M_1$ is a $5$-fold nice product structure of $(G_1,\ca Q_1)$, such that 
$M_1$ has a $\ca Q_1$-aligned tree decomposition $(T_1,\ca X_1)$ of thickness $8$ which exposes~$\ca Q_0$.

Let $\ca Q_2$ denote the family of vertical paths obtained from $\ca Q_1$ by adding the nonempty ones of the paths $Q_7,Q_8,Q_9$
which result by an application of \Cref{clm:6paths}.
Then the subgraph $G_2\subseteq_i G$ induced by $\ca Q_2$ has a $5$-fold nice product structure $P\boxtimes M_2$ and $M_2\supseteq M_1$ 
has a $\ca Q_2$-aligned tree decomposition $(T_2,\ca X_2)$ of thickness $8$ which exposes, for each of the child cycles $D$ of $C$ from \Cref{clm:6paths},
the subset of the paths of $\ca Q_2$ intersecting~$D$.
\end{claim}

\begin{proof}[Subproof.]
The core task is to construct the factor $M_2$ from previous~$M_1$, where we may obviously assume that the path $P$ is suitably long.
Starting from $M:=M_1$, we will use $M$ to denote the intermediate graph $M_1\subseteq M\subseteq M_2$ built so far in the proof.

For $k=7,8,9$ in this order, and assuming $Q_k$ is nonempty, we do the following.
Let $Q_k$ have ends $q_1,q_2\in V(Q_k)$ such that $\ell_T(q_1)\leq\ell_T(q_2)$.
We add to $M$ five new vertices $(Q_k,j)$ where $j\in\{1,\ldots,5\}$ inducing a $5$-clique,
and represent the vertices $v\in V(Q_k)$ in $P\boxtimes M$ by $v'=[p_i,(Q_k,j)]$ where $i=\ell_T(v)$ and $j$ is as follows;
$j=1$ if $v=q_1$, or $j=5$ if $v=q_2\not=q_1$, or $j=2+(i\!\mod3)$ if $v$ is an internal vertex of~$Q_k$.
Notice that our choice ensures fulfillment of both conditions \eqref{it:pmvert} and \eqref{it:threedisti} for $Q=Q_k$.
These vertices $v'=[p_i,(Q_k,j)]$ for $v\in V(Q_k)$ clearly induce a subgraph in $P\boxtimes M$ isomorphic to the path~$Q_k$.

Next, we define all edges between these new vertices $(Q_k,j)$ and the rest of $M$.
We know from \Cref{clm:6paths}\,b) and planarity that no vertex of $V(G_1)$ or of $V(Q_{k'})$ where $7\leq k'<k$
is adjacent to any internal vertex of $Q_k$ (only possibly to $q_1$ or $q_2$).
Hence, there are no more edges in $M$ between $(Q_k,j)$ for $j\in\{2,3,4\}$ and the rest of $M$.

The possible additional neighbours of $(Q_k,1)$ in $M$ are determined as follows.
Pick any $Q\in\ca Q_1\cup\{Q_{k'}:7\leq k'<k\}$
such that $q_1$ is adjacent to a vertex of $V(Q)$.
As stated above, $q_1$ is represented in $P\boxtimes M$ by the vertex $q_1'=[p_i,(Q_k,1)]$.
Furthermore, the possible neighbours of $q_1$ on $Q$ belong to some triple of consecutive vertices $u,v,w\in V(Q)$ which are,
by the condition \eqref{it:threedisti} above, represented in the factor $M$ of $P\boxtimes M$ by some pairwise distinct vertices $m_v=(Q,j)$, $m_u=(Q,j')$ and $m_w=(Q,j'')$.
It is thus sound to demand that the neighbours of $(Q_k,1)$ in $M$ are; $m_v$ iff $q_1v\in E(G)$, $m_u$ iff $q_1u\in E(G)$ and $m_w$ iff $q_1w\in E(G)$.
This is (independently) repeated for all choices of $Q$ being adjacent to~$q_1$.

The possible additional neighbours of $(Q_k,5)$, if $q_2\not=q_1$, in $M$ are determined in the same way as for~$q_1$.
We have hence defined the edge set of $M$, extended with the vertices assigned to the primeval vertical path $Q_k$, such that
the corresponding induced subgraph of $P\boxtimes M$ is isomorphic to the subgraph of $G$ induced by $\ca Q_1\cup\{Q_7,\ldots,Q_k\}$.

Iterating the previous steps up to $k=9$, we obtain the desired factor~$M=M_2$.
It remains to check out a suitable tree decomposition of~$M_2$.
Let $\ca Q_0\subseteq\ca Q_1$ denote the subset of the (at most~$6$) paths of $\ca Q_1$ intersecting~$C$,
and recall that the decomposition $(T_1,\ca X_1)$ of $M_1$ is assumed to expose $\ca Q_0$ in a node $z\in V(T_1)$.
We form $T_2$ from $T_1$ by adding two new mutually adjacent nodes $z_1,z_2$ such that $z_1$ is also adjacent to~$z$, and assigning them new bags 
$Z_1:=\big(\ca Q_0\cup\{Q_7,Q_8\}\big)\times\{1,\ldots,5\}$ and $Z_2:=\big((\ca Q_0\setminus\{Q_h\})\cup\{Q_7,Q_8,Q_9\}\big)\times\{1,\ldots,5\}$
where $Q_h\in\ca Q_0$ is as in \Cref{clm:6paths}\,\ref{it:noQ9}).
Then these new bags expose the subset(s) of $\ca Q_2$ intersecting each of the child cycles $D$ of $C$, as desired, and, 
by \Cref{clm:6paths}\,\ref{it:noQ9}), $(T_2,\ca X_2)$ where $\ca X_2=\ca X_1\cup\{Z_1,Z_2\}$ is a valid tree decomposition of~$M_2$.
Furthermore, $(T_2,\ca X_2)$ is by the definition $\ca Q_2$-aligned of thickness~$8$.
\end{proof}

Finally, we denote by $C_0$ the outer triangular face of~$G$ (containing the root of~$T$), 
and treat the vertices of $C_0$ as three zero-length vertical paths $Q_1,Q_2,Q_3$.
We initially apply \Cref{clm:6paths} via \Cref{clm:6recur} to $C=C_0$ and $\ca Q_1=\{Q_1,Q_2,Q_3\}$, and continue with the same recursively
(the assumptions are clearly satisfied again) for each of the child cycles $D$ of $C$ obtained in \Cref{clm:6paths} until we get $U_D=\emptyset$.
At the end we get, from \Cref{clm:6recur}, a $5$-fold nice product structure $P\boxtimes M$ of whole $G$ where $M$ has a ($\ca Q_2$-)aligned
tree decomposition of thickness~$8$ -- that is, every bag of the decomposition is of size at most $5\cdot8=40$.
Hence, the tree-width of $M$ is at most $40-1=39$.
\end{proof}

\section{Concluding Remarks}
\label{sec:conclu}

The primary focus of our paper is an introduction of a new concept of potential interest,
and as such it naturally brings many questions and open problems, (some of) which we briefly survey in this last section.

From the computer-science perspective, probably the most important question is about the complexity of computing the $\ca H$-clique-width.
Computing traditional clique-width exactly is NP-hard~\cite{DBLP:journals/siamdm/FellowsRRS09}, and hence
the same holds for computing $\ca H$-clique-width exactly in general.
However, a question is whether for some special classes $\ca H$ one could compute exact $\ca H$-clique-width faster.
This is trivially possible, by \Cref{clm:basics}\,\ref{it:H2}), when $\ca H$ is the class of all graphs -- which is uninteresting.
Is it true that computing $\ca H$-clique-width exactly is NP-hard for every fixed family $\ca H$ except in ``similarly trivial'' cases?

On the other hand, traditional clique-width can be approximated in FPT time with respect to the solution value
\cite{DBLP:journals/jct/OumS06,DBLP:journals/siamcomp/HlinenyO08}.
A big goal would be to extend this approximation result to $\ca H$-clique-width, perhaps with an additional parameter
capturing some properties of~$\ca H$.
In particular, with respect to \Cref{sec:induprod}, we emphasise:
\begin{problem}
Let $\ca P^\circ$ denote the class of reflexive paths.
Can one, for input graph $G$, approximate $\hcw{\ca P^\circ\!}(G)$ in FPT time with respect to the solution value?
\end{problem}

Next group of questions concerns combinatorial properties investigated in this paper.
In regard of \Cref{sec:propehcw} and boundedness of local clique-width, we bring the following one:
\begin{problem}[cf.~\Cref{thm:localcw}]
Can one characterise families $\ca H$ of loop graphs such that, for all graphs, their bounded $\ca H$-clique-width implies bounded local clique-width?
\end{problem}

A more interesting and natural question, however, comes in a direct relation to the Planar graph product structure theorem and to \Cref{thm:fromprods}.
We know that graphs of bounded clique-width that do not contain subgraphs isomorphic to large $K_{t,t}$ are as well of bounded tree-width.
A natural counterpart of this claim in the context of $\ca P^\circ$-clique-width would be:
\begin{problem}[cf.~\Cref{thm:hcwproduct}, \Cref{thm:fromprods}]
Assume $t$ is a fixed integer, and $G$ an arbitrary graph such that $\hcw{\ca P^\circ}(G)\leq t$ and $G$ has no subgraph isomorphic to $K_{t,t}$.
Is it then true that $G\subseteq P\boxtimes M$ where $P$ is a path and $M$ is a suitable graph of tree-width bounded in terms of~$t$?
\end{problem}

Another possible question, already mentioned in \Cref{sec:prelim}, is whether $\ca H$-clique-width is in some way closed under taking graph complement. 
This is a prominent and desired property of ordinary clique-width.
At first sight, it could be natural to ask whether, having any simple graph $G$ and its complement $\bar G$, 
we can bound $\hcw{\ca H}(G)$ in terms of $\hcw{\ca H}(\bar G)$.
However, classes $\ca H$ of bounded degree imply bounded local clique-width (\Cref{thm:localcw}) and this property is not even 
asymptotically closed under taking graph complement.

If one instead asks whether, for every graph $H\in\ca H$, there is a graph~$H'$ (perhaps $H'\in\ca H'$ for a suitable other class~$\ca H'$) 
such that for all graphs~$G$, the $\{H'\}$-clique-width of the complement $\bar G$ is bounded by a fixed function of the $\{H\}$-clique-width of $G$,
then the answer is again negative already when $\ca H=\ca P^\circ$ [personal communication].
Though, it is possible to investigate this question for other specifically chosen classes $\ca H$ and~$\ca H'$.

One may also consider, as a direct generalisation of the paths case, investigating graph classes of bounded $\ca H$-clique-width 
when $\ca H=\ca T^\circ$ is the family of (all) reflexive trees.
However, \Cref{prop:squarestar} shows that even if we restrict to the subcase of $\ca H$ formed by all stars and trivial upper bounds, 
we get very rich graph classes; vertices of high degree in $\ca H$ seem to cause a lot of problems in the structural setting.

\smallskip
Lastly, there is a bunch of interesting questions concerning possible relations of $\ca H$-clique-width to the currently hot trend
of studying structural graph properties through the lens of FO logic on graphs and of FO transductions -- transformations
of one graph into another defined by FO formulas.
Since our paper does not define some special terms (e.g., transductions) which are necessary to properly formulate these questions,
we refer interested readers to \cite{HtransdarXiv}.

\subsection*{Acknowledgments}
We acknowledge helpful discussions with Jakub Gajarsk\'y on the questions and problems posed in \Cref{sec:conclu},
and with Adam Straka especially on the question of $\ca H$-clique-width of a graph complement.

\bibliography{Hcw}

\begin{thebibliography}{10}

\bibitem{DBLP:conf/isaac/BekosLH022}
Michael~A. Bekos, Giordano~Da Lozzo, Petr Hlin\v{e}n{\'{y}}, and Michael
  Kaufmann.
\newblock Graph product structure for h-framed graphs.
\newblock In {\em {ISAAC}}, volume 248 of {\em LIPIcs}, pages 23:1--23:15.
  Schloss Dagstuhl - Leibniz-Zentrum f{\"{u}}r Informatik, 2022.
\newblock \href {https://doi.org/10.4230/LIPICS.ISAAC.2022.23}
  {\path{doi:10.4230/LIPICS.ISAAC.2022.23}}.

\bibitem{DBLP:journals/siamdm/BonamyGP22}
Marthe Bonamy, Cyril Gavoille, and Micha{\l} Pilipczuk.
\newblock Shorter labeling schemes for planar graphs.
\newblock {\em {SIAM} J. Discret. Math.}, 36(3):2082--2099, 2022.
\newblock \href {https://doi.org/10.1137/20M1330464}
  {\path{doi:10.1137/20M1330464}}.

\bibitem{DBLP:conf/soda/BonnetGKTW21}
{\'{E}}douard Bonnet, Colin Geniet, Eun~Jung Kim, St{\'{e}}phan Thomass{\'{e}},
  and R{\'{e}}mi Watrigant.
\newblock Twin-width {II:} small classes.
\newblock In {\em {SODA}}, pages 1977--1996. {SIAM}, 2021.
\newblock \href {https://doi.org/10.1137/1.9781611976465.118}
  {\path{doi:10.1137/1.9781611976465.118}}.

\bibitem{DBLP:journals/jacm/BonnetKTW22}
{\'{E}}douard Bonnet, Eun~Jung Kim, St{\'{e}}phan Thomass{\'{e}}, and
  R{\'{e}}mi Watrigant.
\newblock Twin-width {I:} tractable {FO} model checking.
\newblock {\em J. {ACM}}, 69(1):3:1--3:46, 2022.
\newblock \href {https://doi.org/10.1145/3486655} {\path{doi:10.1145/3486655}}.

\bibitem{DBLP:journals/corr/abs-2202-11858}
{\'{E}}douard Bonnet, O{-}joung Kwon, and David~R. Wood.
\newblock Reduced bandwidth: a qualitative strengthening of twin-width in
  minor-closed classes (and beyond).
\newblock {\em CoRR}, abs/2202.11858, 2022.
\newblock URL: \url{https://arxiv.org/abs/2202.11858}.

\bibitem{DBLP:conf/lics/DawarGK07}
Anuj Dawar, Martin Grohe, and Stephan Kreutzer.
\newblock Locally excluding a minor.
\newblock In {\em {LICS}}, pages 270--279. {IEEE} Computer Society, 2007.
\newblock \href {https://doi.org/10.1109/LICS.2007.31}
  {\path{doi:10.1109/LICS.2007.31}}.

\bibitem{DBLP:journals/jct/DingOOV96}
Guoli Ding, Bogdan Oporowski, James~G. Oxley, and Dirk Vertigan.
\newblock Unavoidable minors of large 3-connected binary matroids.
\newblock {\em J. Comb. Theory, Ser. {B}}, 66(2):334--360, 1996.
\newblock \href {https://doi.org/10.1006/JCTB.1996.0026}
  {\path{doi:10.1006/JCTB.1996.0026}}.

\bibitem{distel2023powers}
Marc Distel, Robert Hickingbotham, Michal~T. Seweryn, and David~R. Wood.
\newblock Powers of planar graphs, product structure, and blocking partitions.
\newblock In {\em {EUROCOMB'23: European Conference on Combinatorics, Graph
  Theory and Applications}}, pages 355--361. MUNI Press, Masaryk University,
  2023.
\newblock \href {https://doi.org/10.5817/CZ.MUNI.EUROCOMB23-049}
  {\path{doi:10.5817/CZ.MUNI.EUROCOMB23-049}}.

\bibitem{DBLP:journals/jacm/DujmovicEGJMM21}
Vida Dujmovi\'c, Louis Esperet, Cyril Gavoille, Gwena{\"{e}}l Joret, Piotr
  Micek, and Pat Morin.
\newblock Adjacency labelling for planar graphs (and beyond).
\newblock {\em J. {ACM}}, 68(6):42:1--42:33, 2021.
\newblock \href {https://doi.org/10.1145/3477542} {\path{doi:10.1145/3477542}}.

\bibitem{Dujmovic2020Planar}
Vida Dujmovi\'c, Louis Esperet, Gwena\"el Joret, Bartosz Walczak, and David~R.
  Wood.
\newblock Planar graphs have bounded nonrepetitive chromatic number.
\newblock {\em Advances in Combinatorics}, 3 2020.
\newblock \href {https://doi.org/10.19086/aic.12100}
  {\path{doi:10.19086/aic.12100}}.

\bibitem{DBLP:journals/jacm/DujmovicJMMUW20}
Vida Dujmovi\'c, Gwena{\"{e}}l Joret, Piotr Micek, Pat Morin, Torsten Ueckerdt,
  and David~R. Wood.
\newblock Planar graphs have bounded queue-number.
\newblock {\em J. {ACM}}, 67(4):22:1--22:38, 2020.
\newblock \href {https://doi.org/10.1145/3385731} {\path{doi:10.1145/3385731}}.

\bibitem{DBLP:journals/jctb/DujmovicMW23}
Vida Dujmovi\'c, Pat Morin, and David~R. Wood.
\newblock Graph product structure for non-minor-closed classes.
\newblock {\em J. Comb. Theory, Ser. {B}}, 162:34--67, 2023.
\newblock \href {https://doi.org/10.1016/J.JCTB.2023.03.004}
  {\path{doi:10.1016/J.JCTB.2023.03.004}}.

\bibitem{DBLP:journals/corr/abs-2001-08860}
Zden\v{e}k Dvo\v{r}{\'{a}}k, Tony Huynh, Gwena{\"{e}}l Joret, Chun{-}Hung Liu,
  and David~R. Wood.
\newblock Notes on graph product structure theory.
\newblock In {\em 2019-20 MATRIX Annals}, pages 513--533, Cham, 2021. Springer
  International Publishing.

\bibitem{DBLP:journals/siamdm/FellowsRRS09}
Michael~R. Fellows, Frances~A. Rosamond, Udi Rotics, and Stefan Szeider.
\newblock Clique-width is {NP}-complete.
\newblock {\em {SIAM} J. Discret. Math.}, 23(2):909--939, 2009.
\newblock \href {https://doi.org/10.1137/070687256}
  {\path{doi:10.1137/070687256}}.

\bibitem{DBLP:journals/jacm/FrickG01}
Markus Frick and Martin Grohe.
\newblock Deciding first-order properties of locally tree-decomposable
  structures.
\newblock {\em J. {ACM}}, 48(6):1184--1206, 2001.
\newblock \href {https://doi.org/10.1145/504794.504798}
  {\path{doi:10.1145/504794.504798}}.

\bibitem{DBLP:conf/icalp/HlinenyJ23}
Petr Hlin\v{e}n{\'{y}} and Jan Jedelsk{\'{y}}.
\newblock Twin-width of planar graphs is at most 8, and at most 6 when
  bipartite planar.
\newblock In {\em {ICALP}}, volume 261 of {\em LIPIcs}, pages 75:1--75:18.
  Schloss Dagstuhl - Leibniz-Zentrum f{\"{u}}r Informatik, 2023.
\newblock \href {https://doi.org/10.4230/LIPICS.ICALP.2023.75}
  {\path{doi:10.4230/LIPICS.ICALP.2023.75}}.

\bibitem{HtransdarXiv}
Petr Hlin\v{e}n{\'{y}} and Jan Jedelsk{\'{y}}.
\newblock Transductions of graph classes admitting product structure.
\newblock In {\em 2025 40th Annual ACM/IEEE Symposium on Logic in Computer
  Science (LICS)}, pages 843--855, 2025.
\newblock \href {https://doi.org/10.1109/LICS65433.2025.00069}
  {\path{doi:10.1109/LICS65433.2025.00069}}.

\bibitem{DBLP:journals/siamcomp/HlinenyO08}
Petr Hlin\v{e}n{\'{y}} and Sang{-}il Oum.
\newblock Finding branch-decompositions and rank-decompositions.
\newblock {\em {SIAM} J. Comput.}, 38(3):1012--1032, 2008.
\newblock \href {https://doi.org/10.1137/070685920}
  {\path{doi:10.1137/070685920}}.

\bibitem{DBLP:conf/wg/JacobP22}
Hugo Jacob and Marcin Pilipczuk.
\newblock Bounding twin-width for bounded-treewidth graphs, planar graphs, and
  bipartite graphs.
\newblock In {\em {WG}}, volume 13453 of {\em Lecture Notes in Computer
  Science}, pages 287--299. Springer, 2022.
\newblock \href {https://doi.org/10.1007/978-3-031-15914-5\_21}
  {\path{doi:10.1007/978-3-031-15914-5\_21}}.

\bibitem{DBLP:conf/mfcs/KralPS24}
Daniel Kr{\'{a}}l, Krist{\'{y}}na Pek{\'{a}}rkov{\'{a}}, and Kenny Storgel.
\newblock Twin-width of graphs on surfaces.
\newblock In {\em {MFCS}}, volume 306 of {\em LIPIcs}, pages 66:1--66:15.
  Schloss Dagstuhl - Leibniz-Zentrum f{\"{u}}r Informatik, 2024.
\newblock \href {https://doi.org/10.4230/LIPICS.MFCS.2024.66}
  {\path{doi:10.4230/LIPICS.MFCS.2024.66}}.

\bibitem{DBLP:journals/algorithmica/Lampis12}
Michael Lampis.
\newblock Algorithmic meta-theorems for restrictions of treewidth.
\newblock {\em Algorithmica}, 64(1):19--37, 2012.
\newblock \href {https://doi.org/10.1007/S00453-011-9554-X}
  {\path{doi:10.1007/S00453-011-9554-X}}.

\bibitem{DBLP:journals/jct/OumS06}
Sang{-}il Oum and Paul Seymour.
\newblock Approximating clique-width and branch-width.
\newblock {\em J. Comb. Theory, Ser. {B}}, 96(4):514--528, 2006.
\newblock \href {https://doi.org/10.1016/J.JCTB.2005.10.006}
  {\path{doi:10.1016/J.JCTB.2005.10.006}}.

\bibitem{DBLP:journals/combinatorics/UeckerdtWY22}
Torsten Ueckerdt, David~R. Wood, and Wendy Yi.
\newblock An improved planar graph product structure theorem.
\newblock {\em Electron. J. Comb.}, 29(2), 2022.
\newblock \href {https://doi.org/10.37236/10614} {\path{doi:10.37236/10614}}.

\end{thebibliography}

\end{document}